\documentclass[draft,reqno]{amsproc}
\usepackage{amssymb}
\usepackage{euscript}
\usepackage{mathrsfs} 
\usepackage{units}   
\usepackage{color}

\makeatletter
\@namedef{subjclassname@2010}{%
\textup{2010} Mathematics Subject Classification}
\makeatother

\numberwithin{equation}{section}

\newtheorem{thm}{Theorem}[section]
\newtheorem{cor}[thm]{Corollary}
\newtheorem{lem}[thm]{Lemma}
\newtheorem{pro}[thm]{Proposition}

\newtheorem*{thm*}{Theorem}

\theoremstyle{remark}
\newtheorem{rem}[thm]{Remark}

\newtheorem{outc}[thm]{Conclusion}

\theoremstyle{definition}
\newtheorem{exa}[thm]{Example}
\newtheorem{dfn}[thm]{Definition}

\DeclareMathOperator{\D}{d\hspace{-0.25ex}}

\newcommand*{\ascr}{\mathscr A}
\newcommand*{\borel}[1]{{\mathfrak B}(#1)}
\newcommand*{\bscr}{\mathscr B}
\newcommand*{\cbb}{\mathbb C}

\newcommand*{\card}[1]{\mathrm{card}(#1)}
\newcommand*{\esf}{\mathsf{E}}

\newcommand*{\dz}[1]{{\EuScript D}(#1)}

\newcommand*{\dzn}[1]{{\EuScript D}^\infty(#1)}
\newcommand*{\ee}{\mathcal E}

\newcommand*{\fsf}{\mathsf{F}}
\newcommand*{\Ge}{\geqslant}
\newcommand*{\hh}{\mathcal H}
\newcommand*{\hsf}{\mathsf h}

\newcommand*{\is}[2]{\langle#1,#2\rangle}
\newcommand*{\jd}[1]{\EuScript N(#1)}
\newcommand*{\kk}{\mathcal K}

\newcommand*{\Le}{\leqslant}

\newcommand*{\nbb}{\mathbb N}
\newcommand*{\nbbs}{\nbb_{\rm s}}
\newcommand*{\nsf}{\mathsf{N}}

\newcommand*{\ob}[1]{{\EuScript R}(#1)}

\newcommand*{\rbb}{\mathbb R}

\newcommand*{\smalloplus}{\raise0pt
\hbox{$\scriptscriptstyle \oplus$}}

\newcommand*{\xx}{\mathcal X}
\newcommand*{\zbb}{\mathbb Z}

\begin{document}
   \title[On Unbounded Composition Operators]{On Unbounded Composition Operators
in $L^2$-Spaces}
   \author[P.\ Budzy\'{n}ski]{Piotr Budzy\'{n}ski}
   \address{Katedra Zastosowa\'{n} Matematyki,
Uniwersytet Rolniczy w Krakowie, ul.\ Balicka 253c,
PL-30198 Krak\'ow, Poland}
   \email{piotr.budzynski@ur.krakow.pl}
   \author[Z.\ J.\ Jab{\l}o\'nski]{Zenon Jan
Jab{\l}o\'nski}
   \address{Instytut Matematyki,
Uniwersytet Jagiello\'nski, ul.\ \L ojasiewicza 6,
PL-30348 Kra\-k\'ow, Poland}
   \email{Zenon.Jablonski@im.uj.edu.pl}
   \author[I.\ B.\ Jung]{Il Bong Jung}
   \address{Department of Mathematics, Kyungpook National University, Taegu
702-701 Korea}
   \email{ibjung@knu.ac.kr}
   \author[J.\ Stochel]{Jan Stochel}
\address{Instytut Matematyki, Uniwersytet
Jagiello\'nski, ul.\ \L ojasiewicza 6, PL-30348
Kra\-k\'ow, Poland}
   \email{Jan.Stochel@im.uj.edu.pl}
   \thanks{The research of the first, the second
and the fourth authors was supported by the MNiSzW
(Ministry of Science and Higher Education) grant NN201
546438. The research of the first author was partially
supported by the NCN (National Science Center) grant
DEC-2011/01/D/ST1/05805. The third author was
supported by Basic Science Research Program through
the National Research Foundation of Korea (NRF) grant
funded by the Korea Government (MEST) (2009-0093125).}
   \begin{abstract}
Fundamental properties of unbounded composition
operators in $L^2$-spaces are studied.
Characterizations of normal and quasinormal
composition operators are provided. Formally normal
composition operators are shown to be normal.
Composition operators generating Stieltjes moment
sequences are completely characterized. The unbounded
counterparts of the celebrated Lambert's
characterizations of subnormality of bounded
composition operators are shown to be false. Various
illustrative examples are supplied.
   \end{abstract}
   \maketitle
\section{Introduction}
Composition operators (in $L^2$-spaces over
$\sigma$-finite spaces), which play an essential role
in Ergodic Theory, turn out to be interesting objects
of Operator Theory. The questions of boundedness,
normality, quasinormality, subnormality, hyponormality
etc.\ of such operators have been answered (cf.\
\cite{dun-sch,sin,nor,wh,ha-wh,sin-vel,lam3,lam1,lam2,di-ca,emb-lam4,
emb-lam3,sin-man,vel-pan,bu-ju-la,B-S1,B-S2}; see also
\cite{emb-lam2,ml,sto,da-st,2xSt} for particular
classes of composition operators). This means that the
theory of bounded composition operators on
$L^2$-spaces is well-developed.

The literature on unbounded composition operators in
$L^2$-spaces is meagre. So far, only the questions of
seminormality, k-expansivity and complete
hyperexpansivity have been studied (cf.\
\cite{ca-hor,jab}). Very little is known about other
properties of unbounded composition operators. To the
best of our knowledge, there is no paper concerning
the issue of subnormality of such operators. It is a
difficult question mainly because Lambert's criterion
for subnormality of bounded operators (cf.\
\cite{lam}) is no longer valid for unbounded ones. In
the present paper we show that the unbounded
counterparts of the celebrated Lambert's
characterizations of subnormality of bounded
composition operators given in \cite{lam1} fail to
hold. This is achieved by proving that a composition
operator satisfies the requirements of Lambert's
characterizations if and only if it generates
Stieltjes moment sequences (see Definition
\ref{maindef} and Theorem \ref{gsms}). Thus, knowing
that there exists a non-subnormal composition operator
which generates Stieltjes moment sequences (see
\cite[Theorem 4.3.3]{j-j-s0}), we obtain the
above-mentioned result (see Conclusion \ref{conclud}).
We point out that there exists a non-subnormal
formally normal operator which generates Stieltjes
moment sequences (for details see \cite[Section
3.2]{b-j-j-s}). This is never the case for composition
operators because, as shown in Theorem \ref{fn=n},
each formally normal composition operator is normal,
and as such subnormal. We refer the reader to
\cite{StSz1,StSz2,StSz3,StSz} for the foundations of
the theory of unbounded subnormal operators (for the
bounded case see \cite{hal1,con}).

The above discussion makes plain the importance of the
question of when $C^\infty$-vectors of a composition
operator form a dense subset of the underlying
$L^2$-space. This and related topics are studied in
Section \ref{secprod}. In Section \ref{sbasic}, we
collect some necessary facts on composition operators.
Illustrative examples are gathered in Section
\ref{exampl}. In Section \ref{injs}, we address the
question of injectivity of composition operators. In
Section \ref{ppolar}, we describe the polar
decomposition of a composition operator. Next, in
Sections \ref{naqns} and \ref{fns}, we characterize
normal, quasinormal and formally normal composition
operators. Finally, in Section \ref{gsmss}, we
investigate composition operators which generate
Stieltjes moment sequences. We conclude the paper with
two appendices. In Appendix \ref{apms} we gather
particular properties of $L^2$-spaces exploited
throughout the paper. Appendix \ref{app2} is mostly
devoted to the operator of conditional expectation
which plays an essential role in our investigations.

{\em Caution.} All measure spaces being considered in
this paper, except for Appendices \ref{apms} and
\ref{app2}, are assumed to be $\sigma$-finite.
   \section{Preliminaries}
Denote by $\cbb$, $\rbb$ and $\rbb_+$ the sets of
complex numbers, real numbers and nonnegative real
numbers, respectively. We write $\zbb_+$ for the set
of all nonnegative integers, and $\nbb$ for the set of
all positive integers. The characteristic function of
a subset $\varDelta$ of a set $X$ will be denoted by
$\chi_\varDelta$. We write $\varDelta\vartriangle
\varDelta^\prime= (\varDelta \setminus
\varDelta^\prime) \cup (\varDelta^\prime \setminus
\varDelta)$ for subsets $\varDelta$ and
$\varDelta^\prime$ of $X$. Given a sequence
$\{\varDelta_n\}_{n=1}^\infty$ of subsets of $X$ and a
subset $\varDelta$ of $X$ such that $\varDelta_n
\subseteq \varDelta_{n+1}$ for every $n\in \nbb$, and
$\varDelta = \bigcup_{n=1}^\infty \varDelta_n$, we
write $\varDelta_n \nearrow \varDelta$ (as $n\to
\infty$). Denote by $\card{X}$ the cardinal number of
$X$. If $X$ is a topological space, then $\borel{X}$
stands for the $\sigma$-algebra of Borel subsets of
$X$.

Let $A$ be an operator in a complex Hilbert space
$\hh$ (all operators considered in this paper are
linear). Denote by $\dz{A}$, $\jd{A}$, $\ob{A}$, $\bar
A$ and $A^*$ the domain, the kernel, the range, the
closure and the adjoint of $A$ (in case they exist).
If $A$ is closed and densely defined, then there
exists a unique partial isometry $U$ on $\hh$ such
that $A=U|A|$ and $\jd{U}=\jd{A}$, where $|A|$ stands
for the square root of $A^*A$ (cf.\ \cite[Section
8.1]{b-s}). Set $\dzn{A} = \bigcap_{n=0}^\infty
\dz{A^n}$. Members of $\dzn{A}$ are called
$C^\infty$-vectors of $A$. Denote by
\mbox{$\|\cdot\|_A$} the {\em graph norm} of $A$,
i.e.,
   \begin{align*}
\|f\|_A^2 := \|f\|^2 + \|Af\|^2, \quad f \in \dz{A}.
   \end{align*}
Given $n \in \zbb_+$, we define the norm
$\|\cdot\|_{A,n}$ on $\dz{A^n}$ by
   \begin{align*}
\|f\|_{A,n}^2 := \sum_{j=0}^n \|A^j f\|^2, \quad f \in
\dz{A^n}.
   \end{align*}
Clearly, for every $n \in \nbb$, \mbox{$(\dz{A^n},
\|\cdot\|_{A^n})$} and \mbox{$(\dz{A^n},
\|\cdot\|_{A,n})$} are inner product spaces (with
standard inner products). A vector subspace $\ee$ of
$\dz{A}$ is called a {\em core} for $A$ if $\ee$ is
dense in $\dz{A}$ with respect to the graph norm of
$A$. Denote by $I$ the identity operator on $\hh$.

By applying Propositions \ref{sum-abstr} and
\ref{clos}, one may obtain a criterion for closedness
of a linear combination of composition operators.
   \begin{pro}\label{sum-abstr}
Let $A_1, \ldots, A_n$ be closed operators in $\hh$
$(n\in \nbb)$. Then $\sum_{j=1}^n A_j$ is closed if
and only if there exists $c\in \rbb_+$ such that
   \begin{align} \label{pb}
\sum_{j=1}^n \|A_j f\|^2 \Le c \bigg(\|f\|^2 +
\|\sum_{j=1}^n A_j f\|^2\bigg), \quad f \in
\bigcap_{j=1}^n \dz{A_j}.
   \end{align}
   \end{pro}
   \begin{proof}
Define the vector space $\xx=\bigcap_{j=1}^n \dz{A_j}$
and the norm \mbox{$\|\cdot\|_{*}$} on $\xx$ by
$\|f\|_{*}^2 = \|f\|^2 + \sum_{j=1}^n \|A_j f\|^2$ for
$f \in \xx$. Since the operators $A_1, \ldots, A_n$
are closed, we deduce that
\mbox{$(\xx,\|\cdot\|_{*})$} is a Hilbert space.
Recall that $A:=\sum_{j=1}^n A_j$ is closed if and
only if \mbox{$(\xx,\|\cdot\|_{A})$} is a Hilbert
space. Since the identity map from
\mbox{$(\xx,\|\cdot\|_{*})$} to
\mbox{$(\xx,\|\cdot\|_{A})$} is continuous, we
conclude from the inverse mapping theorem that
\mbox{$(\xx,\|\cdot\|_{A})$} is a Hilbert space if and
only if \eqref{pb} holds for some $c\in \rbb_+$.
   \end{proof}
A densely defined operator $N$ in $\hh$ is said to be
{\em normal} if $N$ is closed and $N^*N=NN^*$ (or
equivalently if and only if $\dz{N}=\dz{N^*}$ and
$\|Nf\|=\|N^*f\|$ for all $f \in \dz{N}$, see
\cite[Proposition, p.\ 125]{weid}). We say that a
densely defined operator $A$ in $\hh$ is {\em formally
normal} if $\dz{A} \subseteq \dz{A^*}$ and
$\|Af\|=\|A^*f\|$ for all $f \in \dz{A}$ (cf.\
\cite{Cod-1,B-C}). A densely defined operator $A$ in
$\hh$ is called {\em hyponormal} if $\dz{A} \subseteq
\dz{A^*}$ and $\|A^*f\| \Le \|Af\|$ for all $f \in
\dz{A}$ (cf.\ \cite{jj3,ot-sch,sz1}). Clearly, a
closed densely defined operator $A$ in $\hh$ is normal
if and only if both operators $A$ and $A^*$ are
hyponormal. It is well-known that normality implies
formal normality and formal normality implies
hyponormality, but none of these implications can be
reversed in general. We say that a densely defined
operator $S$ in $\hh$ is {\em subnormal} if there
exist a complex Hilbert space $\kk$ and a normal
operator $N$ in $\kk$ such that $\hh \subseteq \kk$
(isometric embedding), $\dz{S} \subseteq \dz{N}$ and
$Sf = Nf$ for all $f \in \dz{S}$.

The members of the next class are related to subnormal
operators. A closed densely defined operator $A$ in
$\hh$ is said to be {\em quasinormal} if $A$ commutes
with the spectral measure $E_{|A|}$ of $|A|$, i.e.,
$E_{|A|}(\varDelta) A \subseteq A E_{|A|}(\varDelta)$
for all $\varDelta \in \borel{\rbb_+}$ (cf.\
\cite{bro,StSz1}). In view of \cite[Proposition
1]{StSz1}, a closed densely defined operator $A$ in
$\hh$ is quasinormal if and only if $U |A| \subseteq
|A|U$, where $A=U|A|$ is the polar decomposition of
$A$. This combined with \cite[Theorem 8.1.5]{b-s}
shows that if $A$ is a normal operator, then $A$ is
quasinormal and $\jd{A}=\jd{A^*}$. In turn,
quasinormality together with the inclusion $\jd{A^*}
\subseteq \jd{A}$ characterizes normality. This result
can be found in \cite{StSz-b}. For the reader's
convenience, we include its proof.
   \begin{thm} \label{quasi2n}
An operator $A$ in $\hh$ is normal if and only if $A$
is quasinormal and $\jd{A^*} \subseteq \jd{A}$.
Moreover, if $A$ is normal, then $\jd{A} = \jd{A^*}$.
   \end{thm}
   \begin{proof}
In view of the above discussion it is enough to prove
the sufficiency. First we show that if $A$ is
quasinormal and $A = U|A|$ is its polar decomposition,
then $U |A| = |A|U$. Indeed, by \cite[Proposition
1]{StSz1}, $U |A| \subseteq |A|U$. Taking adjoints, we
get $U^*|A|\subseteq |A|U^*$, which implies that
$U^*(\dz{|A|}) \subseteq \dz{|A|}$. Hence, if $f \in
\dz {|A|U}$, then $U^* Uf\in\dz{|A|}$. Since $I - U^*
U$ is the orthogonal projection of $\hh$ onto
$\jd{|A|}$, we conclude that $f = U^* U f + (I - U^*
U)f \in \dz{|A|}$. This shows that $\dz {|A|U} \subset
\dz{U|A|}$, which implies that $U |A| = |A|U$.

Now suppose that $A$ is quasinormal and $\jd{A^*}
\subseteq \jd{A}$. Since the operators $P := UU^*$ and
$P^\perp := (I-P)$ are the orthogonal projections of
$\hh$ onto $\overline{\ob{A}}$ and $\jd{A^*}$,
respectively, we infer from the inclusion $\jd{A^*}
\subset \jd{A}$ that
   \begin{equation} \label{e8.1b}
\ob{P^\perp} \subset \jd{A}=\jd{|A|} \subset
\dz{|A|^2}.
   \end{equation}
It follows from $U|A|=|A|U$ and $A^* = |A|U^*$ that
   \begin{equation} \label{e8.1c}
AA^* = U|A|^2U^* = |A|^2P.
   \end{equation}
We will show that
   \begin{equation} \label{e8.1d}
|A|^2P = |A|^2.
    \end{equation}
Indeed, if $f\in \hh$, then, by \eqref{e8.1b} and the
equality $f=Pf + P^\perp f$, we see that $Pf\in
\dz{|A|^2}$ if and only if $f \in \dz{|A|^2}$. This
implies that $\dz{|A|^2P} = \dz{|A|^2}$. Using
\eqref{e8.1b} again, we see that $|A|^2 f = |A|^2 Pf$
for every $f \in \dz{|A|^2}$. Hence the equality
\eqref{e8.1d} is valid. Combining \eqref{e8.1c} with
\eqref{e8.1d}, we get $AA^*=A^*A$.

The ``moreover'' part is well-known and easy to prove.
   \end{proof}
Recall that quasinormal operators are subnormal (see
\cite[Theorem 1]{bro} and \cite[Theorem 2]{StSz1}).
The reverse implication does not hold in general.
Clearly, subnormal operators are hyponormal, but not
reversely. It is worth pointing out that formally
normal operators may not be subnormal (cf.\
\cite{Cod,Sch1,sto1}).

A finite complex matrix $[c_{i,j}]_{i,j=0}^n$ is said
to be {\em nonnegative} if
   \begin{align*}
\sum_{i,j=0}^n c_{i,j} \alpha_i \bar\alpha_j \Ge 0,
\quad \alpha_0,\ldots, \alpha_n \in \cbb.
   \end{align*}
If this is the case, then we write
$[c_{i,j}]_{i,j=0}^n \Ge 0$. A sequence
$\{\gamma_n\}_{n=0}^\infty$ of real numbers is said to
be a {\em Stieltjes} moment sequence if there exists a
positive Borel measure $\rho$ on $\rbb_+$ such that
   \begin{align*}
\gamma_n = \int_{\rbb_+} s^n \D \rho(s), \quad n \in
\zbb_+.
   \end{align*}
A sequence $\{\gamma_n\}_{n=0}^\infty \subseteq \rbb$
is said to be {\em positive definite} if for every
$n\in \zbb_+$, $[\gamma_{i+j}]_{i,j=0}^n \Ge 0$. By
the Stieltjes theorem (see \cite[Theorem 6.2.5]{ber}),
we have
   \begin{align} \label{Sti}
   \begin{minipage}{70ex}
a sequence $\{\gamma_n\}_{n=0}^\infty \subseteq \rbb$
is a Stieltjes moment sequence if and only if the
sequences $\{\gamma_n\}_{n=0}^\infty$ and
$\{\gamma_{n+1}\}_{n=0}^\infty$ are positive definite.
   \end{minipage}
   \end{align}
   \begin{dfn}  \label{maindef}
We say that an operator $S$ in $\hh$ {\em generates
Stieltjes moment sequences} if $\dzn{S}$ is dense in
$\hh$ and $\{\|S^n f\|^2\}_{n=0}^\infty$ is a
Stieltjes moment sequence for every $f \in \dzn{S}$.
   \end{dfn}
It is well-known that if $S$ is subnormal, then
$\{\|S^n f\|^2\}_{n=0}^\infty$ is a Stieltjes moment
sequence for every $f \in \dzn{S}$ (see
\cite[Proposition 3.2.1]{b-j-j-s}; see also
Proposition \ref{nth} below). Hence, if $\dzn{S}$ is
dense in $\hh$ and $S$ is subnormal, then $S$
generates Stieltjes moment sequences. It turns out
that the converse implication does not hold in general
(see \cite[Section 3.2]{b-j-j-s}).

The following can be proved analogously to
\cite[Proposition 3.2.1]{b-j-j-s} by using
\eqref{Sti}.
   \begin{pro}\label{nth}
If $S$ is a subnormal operator in $\hh$, then the
following two assertions hold\/{\em :}
   \begin{enumerate}
   \item[(i)]
$\big[\|S^{i+j}f\|^2\big]_{i,j=0}^n \Ge 0$ for all
$f\in \dz{S^{2n}}$ and $n\in \zbb_+$,
   \item[(ii)]
$\big[\|S^{i+j+1}f\|^2\big]_{i,j=0}^n \Ge 0$ for all
$f\in \dz{S^{2n+1}}$ and $n\in \zbb_+$.
   \end{enumerate}
   \end{pro}
%
   For the reader's convenience, we state a theorem
which is occasionally called the Mittag-Leffler
theorem (cf.\ \cite[Lemma 1.1.2]{Sch}).
   \begin{thm} \label{Mit-Lef}
Let $\{\ee_n\}_{n=0}^\infty$ be a sequence of Banach
spaces such that for every $n\in \zbb_+$, $\ee_{n+1}$
is a vector subspace of $\ee_n$, $\ee_{n+1}$ is dense
in $\ee_n$ and the embedding map of $\ee_{n+1}$ into
$\ee_{n}$ is continuous. Then
$\bigcap_{n=0}^\infty\ee_n$ is dense in each space
$\ee_k$,~ $k\in\zbb_+$.
   \end{thm}
   \section{\label{sbasic}Basic properties of composition operators}
From now on, except for Appendices \ref{apms} and
\ref{app2}, $(X, \ascr, \mu)$ always stands for a
$\sigma$-finite measure space. We shall abbreviate the
expressions ``almost everywhere with respect to $\mu$''
and ``for $\mu$-almost every $x$'' to ``a.e.\ $[\mu]$''
and ``for $\mu$-a.e.\ $x$'', respectively. As usual,
$L^2(\mu)=L^2(X,\ascr, \mu)$ denotes the Hilbert space
of all square integrable complex functions on $X$. The
norm and the inner product of $L^2(\mu)$ are denoted by
\mbox{$\|\cdot\|$} and $\is{\cdot}{\mbox{-}}$,
respectively. Let $\phi$ be an $\ascr$-{\em measurable}
transformation\footnote{\;By a transformation of $X$ we
understand a map from $X$ to $X$.} of $X$, i.e.,
$\phi^{-1}(\varDelta) \in \ascr$ for all $\varDelta \in
\ascr$. Denote by $\mu\circ \phi^{-1}$ the positive
measure on $\ascr$ given by $\mu\circ
\phi^{-1}(\varDelta)=\mu(\phi^{-1}(\varDelta))$ for all
$\varDelta \in \ascr$. We say that $\phi$ is {\em
nonsingular} if $\mu\circ \phi^{-1}$ is absolutely
continuous with respect to $\mu$. It is easily seen
that if $\phi$ is nonsingular, then the mapping
$C_\phi\colon L^2(\mu) \supseteq \dz{C_\phi}\to
L^2(\mu)$ given by
   \begin{align} \label{1}
\dz{C_\phi} = \{f \in L^2(\mu) \colon f \circ \phi \in
L^2(\mu)\} \text{ and } C_\phi f = f \circ \phi \text{
for } f \in \dz{C_\phi},
   \end{align}
is well-defined and linear. Such an operator is called
a {\em composition operator} induced by $\phi$; the
transformation $\phi$ will be referred to as a {\em
symbol} of $C_\phi$. Note that if the operator
$C_\phi$ given by \eqref{1} is well-defined, then the
transformation $\phi$ is nonsingular.

{\em Convention.} For the remainder of this paper,
whenever $C_{\phi}$ is mentioned the transformation
$\phi$ is assumed to be nonsingular.

If $\phi$ is nonsingular, then by the Radon-Nikodym
theorem there exists a unique (up to sets of measure
zero) $\ascr$-measurable function $\mathsf \hsf_\phi
\colon X \to [0,\infty]$ such that
   \begin{align} \label{1.5}
\mu\circ \phi^{-1}(\varDelta) = \int_\varDelta
\hsf_\phi \D\mu, \quad \varDelta \in \ascr.
   \end{align}
Here and later on $\phi^n$ stands for the $n$-fold
composition of $\phi$ with itself if $n \Ge 1$ and
$\phi^0$ for the identity transformation of $X$. We
also write $\phi^{-n} (\varDelta) :=
(\phi^n)^{-1}(\varDelta)$ for $\varDelta \in \ascr$
and $n\in \zbb_+$. Note that $\hsf_{\phi^0} =1$ a.e.\
$[\mu]$. It is clear that the composition $\phi_1
\circ \cdots \circ \phi_n$ of finitely many
nonsingular transformations $\phi_1, \ldots, \phi_n$
of $X$ is a nonsingular transformation and
   \begin{align} \label{prnst}
C_{\phi_n} \cdots C_{\phi_1} \subseteq C_{\phi_1\circ
\cdots \circ \phi_n}, \quad n \in \nbb.
   \end{align}

Now we construct an $\ascr$-measurable transformation
$\phi$ of $X$ such that $\phi$ is not nonsingular
while $\phi^2$ is nonsingular.
   \begin{exa}
Set $X=\{0\}\cup\{1\}\cup [2,3]$. Let
$\ascr=\{\varDelta \cap X\colon \varDelta \in
\borel{\rbb_+}\}$. Define the finite Borel measure
$\mu$ on $X$ by
   \begin{align*}
\mu(\varDelta)= \chi_\varDelta(0) + \chi_\varDelta(1)
+ m(\varDelta\cap [2,3]), \quad \varDelta \in \ascr,
   \end{align*}
where $m$ stands for the Lebesgue measure on $\rbb$.
Let $\phi$ be an $\ascr$-measurable transformation of
$X$ given by $\phi(0)=2$, $\phi(1)=1$ and $\phi(x)=1$
for $x \in [2,3]$. Since $\mu(\{2\})=0$ and $(\mu\circ
\phi^{-1})(\{2\})=1$, we see that $\phi$ is not
nonsingular. However, $\phi^2$ is nonsingular because
$\phi^2(x)=1$ for all $x\in X$ and $\mu(\{1\})>0$.
   \end{exa}
Suppose that $\phi$ is a nonsingular transformation of
$X$. In view of the measure transport theorem
(\cite[Theorem C, p.\ 163]{hal2}), we have
   \begin{align}  \label{base}
\int_X |f \circ \phi|^2 \D \mu = \int_X |f|^2
\hsf_\phi \D \mu \text{ for every $\ascr$-measurable
function } f\colon X \to \cbb.
   \end{align}
This implies that
   \begin{gather} \label{2}
\dz{C_\phi} = L^2((1 + \hsf_\phi)\D\mu), \quad
\|f\|_{C_\phi}^2 = \int_X |f|^2 (1 + \hsf_\phi) \D\mu,
   \\       \label{3}
\dz{C_{\phi}^n} = L^2\Big(\Big(\sum_{j=0}^n
\hsf_{\phi^j}\Big)\D\mu\Big), \quad \|f\|_{C_\phi,n}^2
= \int_X |f|^2 \Big(\sum_{j=0}^n \hsf_{\phi^j}\Big)
\D\mu, \quad n \in \zbb_+.
   \end{gather}
Moreover, if $\phi_1, \ldots, \phi_n$ are nonsingular
transformations of $X$ ($n \in \nbb$), then
   \begin{align} \label{prodig}
\dz{C_{\phi_n} \cdots C_{\phi_1}} = L^2((1 +
\sum_{j=1}^n \hsf_{\phi_1\circ \cdots \circ
\phi_j})\D\mu).
   \end{align}
The following proposition is somewhat related to
\cite[p.\ 664]{dun-sch} and \cite[Lemma 6.1]{ca-hor}.
   \begin{pro} \label{clos}
Let $\phi$ be a nonsingular transformation of $X$.
Then $C_\phi$ is a closed operator and
   \begin{align} \label{clos2}
\text{$\overline{\dz{C_\phi}} = \chi_{\fsf_{\phi}}
L^2(\mu)$ with $\mathsf{F}_{\phi} = \big\{x \in
X\colon \hsf_\phi(x) < \infty\big\}$.}
   \end{align}
Moreover, the following conditions are equivalent{\em
:}
   \begin{enumerate}
   \item[(i)] $C_\phi$ is densely defined,
   \item[(ii)] $\hsf_\phi < \infty$ a.e.\ $[\mu]$,
   \item[(iii)] the measure $\mu\circ \phi^{-1}$
is $\sigma$-finite.
   \end{enumerate}
   \end{pro}
   \begin{proof}
Applying \eqref{2}, we get $C_\phi =
\overline{C_\phi}$ and $\overline{\dz{C_\phi}}
\subseteq \chi_{\fsf_{\phi}} L^2(\mu)$. To prove the
opposite inclusion $\chi_{\fsf_{\phi}}
L^2(\mu)\subseteq \overline{\dz{C_\phi}}$, take $f\in
L^2(\mu)$ such that $f|_{X \setminus \fsf_\phi} = 0$
a.e.\ $[\mu]$, and set $X_{n}=\{x \in X\colon
\hsf_\phi(x) \Le n\}$ for $n \in \nbb$. Noting that
$X_n \nearrow \fsf_\phi$ as $n\to \infty$, we see that
$\int_{X} |\chi_{X_{n}}f|^2 (1+\hsf_\phi) \D \mu <
\infty$ for all $n\in \nbb$, and $\lim_{n\to
\infty}\int_X|f - \chi_{X_n}f|^2 \D\mu=0$, which
completes the proof of \eqref{clos2}.

(i)$\Leftrightarrow$(ii) Employ \eqref{clos2}.

(ii)$\Leftrightarrow$(iii) Apply \eqref{1.5} and the
assumption that $\mu$ is $\sigma$-finite.
   \end{proof}
   \begin{cor}
Suppose that $\phi_1, \ldots, \phi_n$ are nonsingular
transformations of $X$ and $\lambda_1, \ldots,
\lambda_n$ are nonzero complex numbers
$($$n\in\nbb$$)$. Then $\sum_{j=1}^n \lambda_j
C_{\phi_j}$ is densely defined if and only if
$C_{\phi_k}$ is densely defined for every $k=1,
\ldots, n$.
   \end{cor}
   \begin{proof}
By \eqref{2}, ${\EuScript D}(\sum_{j=1}^n \lambda_j
C_{\phi_j}) = L^2((1 + \sum_{j=1}^n \hsf_{\phi_j})\D
\mu)$, and thus the ``if'' part follows from
Proposition \ref{clos} and Lemma \ref{l1.11}. The
``only if'' part is obvious.
   \end{proof}
   \section{\label{secprod}Products of composition operators}
First we give necessary and sufficient conditions for
a product of composition operators to be densely
defined.
   \begin{pro} \label{prod}
Let $\phi_1, \ldots, \phi_n$ be nonsingular
transformations of $X$ $(2 \Le n < \infty)$. Then the
following assertions hold\/{\em :}
   \begin{enumerate}
   \item[(i)] $C_{\phi_n}
\cdots C_{\phi_1}$ is a closable operator,
   \item[(ii)] $C_{\phi_n} \cdots C_{\phi_1}$  is
densely defined if and only if $C_{\phi_1\circ \cdots
\circ \phi_k}$ is densely defined for every $k=1,
\ldots, n$,
   \item[(iii)] if $C_{\phi_{n-1}} \cdots C_{\phi_1}$ is densely
defined, then
   \begin{align}  \label{numerek}
C_{\phi_1\circ \cdots \circ \phi_k} =
\overline{C_{\phi_k} \cdots C_{\phi_1}}, \quad k=1,
\ldots, n,
   \end{align}
   \item[(iv)] if $C_{\phi_1 \circ \cdots \circ \phi_n}$
is densely defined, then so is the operator
$C_{\phi_n}$,
   \item[(v)] if $C_{\phi_n} \cdots
C_{\phi_1}$ is densely defined, then so are the
operators $C_{\phi_1}$, \ldots, $C_{\phi_n}$.
   \end{enumerate}
   \end{pro}
   \begin{proof}
(i) Apply \eqref{prnst} and Proposition \ref{clos}.

(ii) To prove the ``if'' part, assume that
$C_{\phi_1\circ \cdots \circ \phi_k}$ is densely
defined for $k=1, \ldots, n$. It follows from
Proposition \ref{clos} that $\hsf_{\phi_1\circ \cdots
\circ \phi_k} < \infty$ a.e.\ $[\mu]$ for $k=1,\ldots,
n$. Applying \eqref{prodig} and Lemma \ref{l1.11} to
$\rho_1\equiv 1$ and $\rho_2 = 1 + \sum_{j=1}^n
\hsf_{\phi_1\circ \cdots \circ \phi_j}$ we get
$\overline{\dz{C_{\phi_n} \cdots
C_{\phi_1}}}=L^2(\mu)$. The ``only if'' part follows
from \eqref{prnst} and the fact that the operators
$C_{\phi_k} \cdots C_{\phi_1}$, $k=1, \ldots, n$, are
densely defined.

(iii) It follows from (ii) and Proposition \ref{clos}
that $h:=\sum_{j=1}^{n-1} \hsf_{\phi_1\circ \cdots
\circ \phi_j} < \infty$ a.e.\ $[\mu]$. Set $Y=\{x\in X
\colon \hsf_{\phi_1\circ \cdots \circ \phi_n} (x) <
\infty\}$ and $\ascr_Y = \{\varDelta \in \ascr\colon
\varDelta \subseteq Y\}$. Equip $\dz{C_{\phi_1\circ
\cdots \circ \phi_n}}$ with the graph norm of
$C_{\phi_1\circ \cdots \circ \phi_n}$ and note that the
mapping
   \begin{align*}
\varTheta\colon \dz{C_{\phi_1\circ \cdots \circ
\phi_n}} \ni f \longmapsto f|_{Y} \in
L^2\big(Y,\ascr_Y,(1+\hsf_{\phi_1\circ \cdots \circ
\phi_n})\D \mu\big)
   \end{align*}
is a well-defined unitary isomorphism (use \eqref{2}).
It follows from Lemma \ref{l1.11} that
$L^2\big(Y,\ascr_Y,(1+ h + \hsf_{\phi_1\circ \cdots
\circ \phi_n})\D \mu\big)$ is dense in
$L^2\big(Y,\ascr_Y,(1+\hsf_{\phi_1\circ \cdots \circ
\phi_n})\D \mu\big)$. Since, by \eqref{prnst} and
\eqref{prodig}, $\varTheta(\dz{C_{\phi_{n}} \cdots
C_{\phi_1}})=L^2\big(Y,\ascr_Y,(1+ h +
\hsf_{\phi_1\circ \cdots \circ \phi_n})\D \mu\big)$,
we deduce that $\overline{C_{\phi_n} \cdots
C_{\phi_1}} = C_{\phi_1\circ \cdots \circ \phi_n}$.
Applying the previous argument to the systems
$(C_{\phi_1}, \ldots, C_{\phi_k})$, $k\in \{1, \ldots,
n-1\}$, we obtain \eqref{numerek}.

   (iv) It is sufficient to discuss the case of $n=2$.
Suppose that $C_{\phi_1\circ \phi_2}$ is densely
defined. In view of Proposition \ref{clos}, the
measure $\mu\circ (\phi_1 \circ \phi_2)^{-1}$ is
$\sigma$-finite. Since $\mu\circ (\phi_1 \circ
\phi_2)^{-1} = (\mu\circ \phi_2^{-1})\circ
\phi_1^{-1}$, we see that the measure $\mu\circ
\phi_2^{-1}$ is $\sigma$-finite as well. Applying
Proposition \ref{clos} again, we conclude that
$C_{\phi_2}$ is densely defined.

(v) Apply (ii) and (iv).
   \end{proof}
   \begin{cor}\label{n-1}
If $C_{\phi}^{n-1}$ is densely defined for some $n\in
\nbb$, then $\overline{C_{\phi}^n} = C_{\phi^n}$.
   \end{cor}
The following is an immediate consequence of
\eqref{prodig} and Corollary \ref{l1.12}.
   \begin{pro} \label{indz}
If $\phi_1, \ldots, \phi_m$ and $\psi_1, \ldots,
\psi_n$ are nonsingular transformations of $X$, then
$\dz{C_{\phi_n} \cdots C_{\phi_1}} \subseteq
\dz{C_{\psi_m} \cdots C_{\psi_1}}$ if and only if
there exists $c \in \rbb_+$ such that $\sum_{j=1}^m
\hsf_{\psi_1\circ \cdots \circ \psi_j} \Le c \big(1+
\sum_{j=1}^n \hsf_{\phi_1\circ \cdots \circ
\phi_j}\big)$ a.e.\ $[\mu]$.
   \end{pro}
Now we give necessary and sufficient conditions for a
product of composition operators to be closed.
   \begin{pro} \label{bjs}
Let $\phi_1, \ldots, \phi_n$ be nonsingular
transformations of $X$ $(2 \Le n < \infty)$. Then the
following three conditions are equivalent{\em :}
   \begin{enumerate}
   \item[(i)] $C_{\phi_n} \cdots
C_{\phi_1} = C_{\phi_1\circ \cdots \circ \phi_n}$,
   \item[(ii)] $\dz{C_{\phi_1\circ \cdots \circ
\phi_n}} \subseteq \dz{C_{\phi_n} \cdots C_{\phi_1}}$,
   \item[(iii)] there exists $c\in \rbb_+$ such that
$\sum_{j=1}^{n-1} \hsf_{\phi_1\circ \cdots \circ
\phi_j} \Le c (1+\hsf_{\phi_1\circ \cdots \circ
\phi_n})$ a.e.\ $[\mu]$.
   \end{enumerate}
Moreover, any of the conditions {\em (i)} to {\em
(iii)} implies that
   \begin{enumerate}
   \item[(iv)] $C_{\phi_n} \cdots
C_{\phi_1}$ is closed.
   \end{enumerate}
If $C_{\phi_{n-1}} \cdots C_{\phi_1}$ is densely
defined, then all the conditions {\em (i)} to {\em
(iv)} are equivalent.
   \end{pro}
   \begin{proof}
The equivalence of (i) and (ii) is a direct
consequence of \eqref{prnst}. The equivalence of (ii)
and (iii) follows from Proposition \ref{indz}. That
(i) implies (iv) follows from Proposition \ref{clos}.
Finally, if the product $C_{\phi_{n-1}} \cdots
C_{\phi_1}$ is densely defined, then (iv) implies (i)
due to Proposition \ref{prod}\,(iii).
   \end{proof}
   \begin{cor}\label{pr1}
If $\phi$ is a nonsingular transformation of $X$, then
the following assertions hold for all $n \in \nbb${\em
:}
   \begin{enumerate}
   \item[(i)] $C_{\phi^n}$ is densely defined if and
only $\hsf_{\phi^n} < \infty$ a.e.\ $[\mu]$,
   \item[(ii)] $C_{\phi}^n$ is densely defined if and
only $\sum_{j=1}^n \hsf_{\phi^j} < \infty$ a.e.\
$[\mu]$,
   \item[(iii)]  $C_{\phi}^n=C_{\phi^n}$ if and only
if there exists $c\in \rbb_+$ such that $\hsf_{\phi^k}
\Le c (1 + \hsf_{\phi^n})$ a.e.\ $[\mu]$ for $k=1,
\ldots, n$.
   \end{enumerate}
   \end{cor}
   \begin{proof}
Use Propositions \ref{clos}, \ref{prod}\,(ii) and
\ref{bjs} (for (ii) see also \cite[p.\ 515]{jab}).
   \end{proof}
   \begin{cor} \label{xn2x}
If $\phi$ is a nonsingular transformation of $X$ and
$\overline{\dz{C_\phi^m}}=L^2(\mu)$ for some $m \in
\nbb$, then there exists a sequence
$\{X_n\}_{n=1}^\infty \subseteq \ascr$ such that
   \begin{enumerate}
   \item[(i)] $X_n\nearrow X$ as $n\to \infty$,
   \item[(ii)] $\mu(X_n) < \infty$ for all $n\in
\nbb$,
   \item[(iii)] $\sum_{j=1}^m \hsf_{\phi^j}(x)
\Le n$ for $\mu$-a.e.\ $x \in X_n$ and $n\in \nbb$.
   \end{enumerate}
   \end{cor}
The question of when $C^\infty$-vectors of an operator
$A$ in a Hilbert space $\hh$ form a dense subspace of
$\hh$ is of independent interest (cf.\
\cite{Sch0,Ko-Th}). If every power of $A$ is densely
defined, then one could expect that $\dzn{A}$ is dense
in $\hh$. This is the case for any closed densely
defined operator (even in a Banach space), the
resolvent set of which is nonempty\footnote{\;This can
be deduced from the fact that the intersection of
ranges of all powers of a bounded operator which has
dense range is dense in the underlying space.}. As
shown below, this is also the case for composition
operators. However, this seems to be not true in
general. Dropping the assumption of closedness, we can
provide a simple counterexample. Indeed, take an
infinite dimensional separable Hilbert space $\hh$.
Then there exists a dense subset $\{e_n\colon n
\in\zbb_+\}$ of $\hh$ which consists of linearly
independent vectors. Let $A$ be the operator in $\hh$
whose domain is the linear span of $\{e_n\colon n
\in\nbb\}$ and $Ae_j = e_{j-1}$ for every $j\in \nbb$.
Since $\{e_n\colon n \Ge k\}$ is dense in $\hh$ for
every $k \in \zbb_+$, we deduce that the operator
$A^n$ is densely defined for every $n \in \zbb_+$.
However, $\dzn{A}=\{0\}$.
   \begin{thm} \label{Mittag}
If $\phi$ is a nonsingular transformation of $X$, then
the following conditions are equivalent\/{\em :}
   \begin{enumerate}
   \item[(i)] $\dz{C_\phi^n}$ is dense in $L^2(\mu)$ for every $n
\in \nbb$,
   \item[(ii)] $\dzn{C_\phi}$ is dense in $L^2(\mu)$,
   \item[(iii)] $\dzn{C_\phi}$  is a core for $C_\phi^n$
for every $n\in \zbb_+$,
   \item[(iv)] $\dzn{C_\phi}$
is dense in $(\dz{C_\phi^n},\|\cdot\|_{C_\phi,n})$ for
every $n\in \zbb_+$.
   \end{enumerate}
   \end{thm}
   \begin{proof}
The implications (iv)$\Rightarrow$(iii),
(iii)$\Rightarrow$(ii) and (ii)$\Rightarrow$(i) are
obvious.

(i)$\Rightarrow$(iv) In view of Corollary
\ref{pr1}\,(ii), $0 \Le \hsf_{\phi^n} < \infty$ a.e.\
$[\mu]$ for all $n\in \nbb$. Given $n \in \zbb_+$ we
denote by $\hh_n$ the inner product space
$(\dz{C_\phi^n}, \|\cdot\|_{C_\phi,n})$. It follows
from \eqref{3} that $\hh_n$ is a Hilbert space which
coincides with $L^2((\sum_{j=0}^n
\hsf_{\phi^j})\D\mu)$. Hence, in view of Lemma
\ref{l1.11}, $\hh_{n+1}$ is a dense subspace of
$\hh_{n}$. Clearly, the embedding map of $\hh_{n+1}$
into $\hh_n$ is continuous. Applying Theorem
\ref{Mit-Lef} to the sequence $\{\hh_n\}_{n=0}^\infty$,
we conclude that $\dzn{C_\phi} = \bigcap_{i=0}^\infty
\hh_i$ is dense in $\dz{C_\phi^n}$ with respect to the
norm $\|\cdot\|_{C_\phi,n}$ for every $n\in \zbb_+$.
This completes the proof.
   \end{proof}
Regarding Theorem \ref{Mittag}, we mention the
following surprising fact which can be deduced from
\cite[Theorem 4.5]{Sch0} by using Theorem
\ref{Mit-Lef} and \cite[Corollaries 1.2 and
1.4]{Sch0}.
   \begin{thm}
Let $A$ be an unbounded selfadjoint operator in a
complex Hilbert space $\hh$ and let $\mathfrak N$ be a
$($possibly empty$)$ subset of $\nbb\setminus \{1\}$
such that $\nbb \setminus \mathfrak N$ is infinite.
Then there exists a closed symmetric operator $T$ in
$\hh$ such that $T \subseteq A$, $\dzn{T}$ is dense in
$\hh$ and for every $k\in\nbb$, $\dzn{T}$ is a core
for $T^k$ if and only if $k \in \nbb \setminus
\mathfrak N$.
   \end{thm}
   \section{Examples} \label{exampl}
We begin by showing that Corollary \ref{n-1} is no
longer true if the assumption that $C_{\phi}^{n-1}$ is
densely defined is dropped.
   \begin{exa} \label{numerek-c}
We will demonstrate that there is a nonsingular
transformation $\phi$ such that $C_{\phi}$ is densely
defined, $C_{\phi^j}$ and $C_{\phi}^j$ are not densely
defined for every $j \in \{2,3, \ldots\}$, and
$\overline{C_{\phi}^3} \varsubsetneq C_{\phi^3}$
(however, by Corollary \ref{n-1},
$\overline{C_{\phi}^2} = C_{\phi^2}$). For this, we
will re-examine Example 4.2 given in \cite{jab}.
Suppose that $\{a_i\}_{i=0}^\infty$,
$\{b_i\}_{i=0}^\infty$ and
$\{c_{i,j}\}_{i,j=0}^\infty$ are disjoint sets of
distinct elements. Set $X=\{a_i\}_{i=0}^\infty \cup
\{b_i\}_{i=0}^\infty \cup \{c_{i,j}\}_{i,j=0}^\infty$
and $\ascr=2^X$. Let $\mu$ be a unique $\sigma$-finite
measure on $\ascr$ determined by
   \begin{align*}
\mu\big(\{x\}\big) =
   \begin{cases}
   1 & \text{if } x=a_i \hspace{1.1ex}\text{ for some
   } i\in \zbb_+,
   \\
   \frac{1}{2^{i+1}} & \text{if } x=b_i \hspace{1.3ex}
   \text{ for some } i\in \zbb_+,
   \\
   \frac{1}{2^{j+1}} & \text{if } x=c_{i,j} \text{ for
   some } i,j\in \zbb_+.
   \end{cases}
   \end{align*}
Define a nonsingular transformation $\phi$ of $X$ by
   \begin{align*}
\phi(x)=
   \begin{cases}
   a_{i+1} & \text{if } x=a_i \hspace{1.1ex} \text{
   for some } i\in \zbb_+,
   \\
   a_0 & \text{if } x=b_i \hspace{1.3ex} \text{ for
   some } i\in \zbb_+,
   \\
   b_i & \text{if } x=c_{i,j} \text{ for some } i,j\in
   \zbb_+.
   \end{cases}
   \end{align*}
Then $\hsf_\phi < \infty$ a.e.\ $[\mu]$, and thus by
Proposition \ref{clos} the operator $C_\phi$ is
densely defined. Since $\hsf_{\phi^2}(a_0)=\infty$, we
infer from Proposition \ref{clos} that $C_{\phi^2}$ is
not densely defined. It follows from \eqref{prodig}
that $\dz{C_{\phi}^3} = L^2((1 + \hsf_\phi +
\hsf_{\phi^2} + \hsf_{\phi^3})\D\mu)$. This and
$\hsf_{\phi^2}(a_0)=\infty$ imply that $f(a_0)=0$ for
every $f \in \dz{C_{\phi}^3}$. Since the convergence
in the graph norm is stronger than the pointwise
convergence, we deduce that $f(a_0)=0$ for every $f
\in \dz{\overline{C_{\phi}^3}}$. As $\dz{C_{\phi^3}} =
L^2((1 + \hsf_{\phi^3})\D\mu)$ (cf.\ \eqref{2}) and
$\hsf_{\phi^3}(a_0)=0$ (because
$\phi^{-3}(\{a_0\})=\varnothing$), we see that
$\chi_{\{a_0\}} \in \dz{C_{\phi^3}} \setminus
\dz{\overline{C_{\phi}^3}}$. Finally, arguing as above
and using the fact that
$\hsf_{\phi^{j+2}}(a_j)=\infty$ for every $j\in
\zbb_+$, we conclude that $C_{\phi^j}$ is not densely
defined for every $j\in \{2,3, \ldots\}$. As a
consequence, $C_{\phi}^j$ is not densely defined for
every $j \in \{2,3, \ldots\}$.
   \end{exa}
The composition operator $C_\phi$ constructed in
Example \ref{numerek-c} is densely defined, its square
is not densely defined, however $\dim \dz{C_\phi^n}=
\infty$ for all $n\in \nbb$ (because $\chi_{\{a_{i}\}}
\in \dz{C_\phi^n}$ for all $i\Ge n-1$). In fact, there
are more pathological examples.
   \begin{exa}
It was proved in \cite[Theorem 4.2]{j-j-s3} that there
exists a hyponormal weighted shift $S$ on a rootless
and leafless directed tree with positive weights whose
square has trivial domain. By \cite[Lemma
4.3.1]{j-j-s0}, $S$ is unitarily equivalent to a
composition operator $C$. As a consequence, $C$ is
injective and hyponormal, and $\dz{C^2} =
\dzn{C}=\{0\}$ (see also \cite{budz} for a recent
construction).
   \end{exa}
Regarding Proposition \ref{prod}, we note that it may
happen that the operators $C_{\phi_1}$ and
$C_{\phi_2}$ are densely defined, while the operators
$C_{\phi_1\circ \phi_2}$ and $C_{\phi_2}C_{\phi_1}$
are not (even if $\phi_1=\phi_2$, see Example
\ref{numerek-c}). Below we will show that for some
$\phi_1$ and $\phi_2$ the composition operator
$C_{\phi_1\circ \phi_2}$ is densely defined (even
bounded), while $C_{\phi_1}$ is not.
   \begin{exa}
Set $X=\zbb_+$ and $\ascr=2^X$. Let $\mu$ be the
counting measure on $X$ and let $\phi_1$ and $\phi_2$
be the nonsingular transformations of $X$ given by
$\phi_1(2n)=n$, $\phi_1(2n+1)=0$ and $\phi_2(n)=2n$
for $n\in \zbb_+$. Then $\phi_1\circ\phi_2$ is the
identity transformation of $X$, and hence
$C_{\phi_1\circ \phi_2}$ is the identity operator on
$L^2(\mu)$. However, since $\mu(\phi_1^{-1}(\{0\})) =
\infty$, the measure $\mu\circ \phi_1^{-1}$ is not
$\sigma$-finite, and thus by Proposition \ref{clos}
the operator $C_{\phi_1}$ is not densely defined.
   \end{exa}
   Our next aim is to provide examples showing that
the equality $C_{\phi}^n=C_{\phi^n}$ which appears in
Corollary \ref{pr1}\,(iii) does not hold in general
even if $\dzn{C_\phi}$ is dense in $L^2(\mu)$ (which
is not the case for the operator given in Example
\ref{numerek-c}).
   \begin{exa}
Set $X=\nbb$ and $\ascr=2^X$. Let $\mu$ be a counting
measure on $X$ and let $\{J_n\}_{n=1}^\infty$ be a
partition of $X$. Define a nonsingular transformation
$\phi$ of $X$ by $\phi(x) = \min J_{n^2}$ for $x \in
J_n$ and $n\in \nbb$. Set $\nbbs = \{n^2\colon n \in
\nbb\}$ and note that
   \begin{align} \label{eqrel}
X=\{1\} \sqcup \bigsqcup_{q\in \nbb\setminus \nbbs}
\big\{q^{2^n}\colon n \in \zbb_+\big\},
   \end{align}
where all terms in \eqref{eqrel} are pairwise disjoint
(they are equivalence classes under the equivalence
relation $\sim$ given by:\ $p \sim q$ if and only if
$p^{2^m} = q^{2^n}$ for some $m,n\in \zbb_+$). Since
$\hsf_{\phi^j}(x) = \card{\phi^{-j}(\{x\})}$ for $x
\in X$ and $j\in \nbb$, we infer from \eqref{eqrel}
that for all $j\in \nbb$ and $x\in X$ ($m$ appearing
below varies over the set of integers)
   \begin{align} \label{invj}
\hsf_{\phi^j}(x) =
   \begin{cases}
   \card{J_1} & \text{ if } x=\min J_1,
   \\
   \card{J_{q^{2^{m-j}}}} & \text{ if } x=\min
J_{q^{2^{m}}} \text{ with } q \in\nbb \setminus \nbbs
\text{ and } m\Ge j,
   \\
   0 & \text{ otherwise}.
   \end{cases}
   \end{align}
By \eqref{eqrel}, \eqref{invj}, Proposition \ref{clos}
and Theorem \ref{Mittag}, the following are
equivalent:
   \begin{itemize}
   \item $\card{J_k} < \aleph_0$ for every $k\in
\nbb$,
   \item $C_\phi$ is densely defined,
   \item $C_\phi^n$ is densely defined for some $n\in
\nbb$,
   \item $C_\phi^n$ is densely defined for every $n\in
\nbb$,
   \item $\dzn{C_\phi}$ is dense in
$L^2(\mu)$.
   \end{itemize}
It follows from \eqref{invj} and Proposition \ref{bjs}
that for a given integer $n\Ge 2$, $C_{\phi}^n$ is
closed if and only if there exists $c\in \rbb_+$ such
that
   \begin{align*}
\card{J_{q^{2^s}}} & \Le c, \quad s=0, \ldots, n-2,\,
q \in \nbb \setminus \nbbs,
   \\
\card{J_{q^{2^{s+1}}}} & \Le c \big(1+
\card{J_{q^{2^{s}}}}\big), \quad s\in \zbb_+,\, q \in
\nbb \setminus \nbbs.
   \end{align*}
Using this and an induction argument, one can prove
that either $C_\phi^n$ is closed for every integer
$n\Ge 1$, or $C_\phi^n$ is not closed for every
integer $n\Ge 2$. Summarizing, if we choose a
partition $\{J_i\}_{i=1}^\infty$ of $X$ such that
$J_n$ is finite for every $n\in \nbb$, and $\sup\{
\card{J_q} \colon q\in \nbb\setminus \nbbs\} =
\aleph_0$ (which is possible), then $\dzn{C_\phi}$ is
dense in $L^2(\mu)$ and $C_\phi^n$ is not closed for
every integer $n\Ge 2$. On the other hand, if
$\kappa\Ge 2$ is any fixed integer and a partition
$\{J_i\}_{i=1}^\infty$ of $X$ is selected so that
$J_1$ is finite and $\card{J_{q^{2^{n}}}}=\kappa^n$
for all $n\in\zbb_+$ and $q \in \nbb\setminus \nbbs$
(which is also possible), then $\dzn{C_\phi}$ is dense
in $L^2(\mu)$ and $C_\phi^n$ is closed and unbounded
for every $n\in \nbb$.
   \end{exa}
   \section{\label{injs}Injectivity of $C_\phi$}
In this section we provide necessary and sufficient
conditions for a composition operator to be injective.
The following set plays an important role in our
considerations.
   \begin{align*}
\mathsf{N}_\phi=\{x \in X \colon \hsf_\phi(x)=0\}.
   \end{align*}
The following description of the kernel of $C_\phi$
follows immediately from \eqref{base}.
   \begin{pro}\label{nullspace}
If $\phi\colon X \to X$ is nonsingular, then
$\jd{C_\phi} = \chi_{\nsf_\phi} L^2(\mu)$.
   \end{pro}
   \begin{pro}\label{kercf}
Let $\phi$ be a nonsingular transformation of $X$.
Consider the following four conditions{\em :}
   \begin{enumerate}
   \item[(i)] $\jd{C_\phi} = \{0\}$,
   \item[(ii)] $\mu(\nsf_\phi)=0$,
   \item[(iii)] $\chi_{\nsf_\phi} \circ \phi =
\chi_{\nsf_\phi}$ a.e.\ $[\mu]$,
   \item[(iv)] $\jd{C_\phi} \subseteq \jd{C_\phi^*}$.
   \end{enumerate}
Then the conditions {\em (i)}, {\em (ii)} and {\em
(iii)} are equivalent. Moreover, if $C_\phi$ is
densely defined, then the conditions {\em (i)} to {\em
(iv)} are equivalent.
   \end{pro}
   \begin{proof}
(i)$\Leftrightarrow$(ii) Apply Proposition
\ref{nullspace} and the $\sigma$-finiteness of $\mu$.

(ii)$\Rightarrow$(iii) Since $\phi$ is nonsingular, we
have $\mu(\nsf_\phi)=0$ and
$\mu(\phi^{-1}(\nsf_\phi))=0$, which implies that
$\mu(\nsf_\phi \vartriangle \phi^{-1}(\nsf_\phi))=0$.
The latter is equivalent to (iii).

(iii)$\Rightarrow$(ii) By the measure transport
theorem, we have
   \begin{align*}
\mu(\nsf_\phi) = \int_X \chi_{\nsf_\phi} \D \mu =
\int_X \chi_{\nsf_\phi} \circ \phi \D \mu = \int_X
\chi_{\nsf_\phi} \hsf_\phi \D \mu =0.
   \end{align*}

Now suppose that $C_\phi$ is densely defined.

(i)$\Rightarrow$(iv) Obvious.

(iv)$\Rightarrow$(ii) Let $\{X_n\}_{n=1}^\infty$ be as
in Corollary \ref{xn2x} (with $m=1$). Then, by
 \eqref{base}, we see that $\chi_{X_n},
\chi_{\nsf_\phi \cap X_n} \in \dz{C_\phi}$ and
$\|C_\phi (\chi_{\nsf_\phi \cap X_n})\|^2 =
\int_{\nsf_\phi \cap X_n} \hsf_\phi \D\mu = 0$ for all
$n \in \nbb$, which together with our assumption that
$\jd{C_\phi} \subseteq \jd{C_\phi^*}$ yields
   \begin{align*}
0 = \is{\chi_{\nsf_\phi \cap X_n}}{C_\phi \chi_{X_n}}
= \int_{\nsf_\phi \cap X_n}\chi_{X_n} \circ \phi \D\mu
= \mu(\nsf_\phi \cap X_n \cap \phi^{-1}(X_n))
   \end{align*}
for all $n \in \nbb$. Since $\nsf_\phi \cap X_n \cap
\phi^{-1}(X_n) \nearrow \nsf_\phi$ as $n\to \infty$,
the continuity of measure implies that
$\mu(\nsf_\phi)=0$. This completes the proof.
   \end{proof}
   \begin{cor} \label{hipnul}
If $C_\phi$ is hyponormal, then $\jd{C_\phi} = \{0\}$.
   \end{cor}
   \begin{proof}
It follows from the definition of hyponormality that
$\jd{C_\phi} \subseteq \jd{C_\phi^*}$. This and
Proposition \ref{kercf} complete the proof.
   \end{proof}
   \begin{cor} \label{fnnul}
If $C_\phi$ is formally normal, then
   \begin{align*}
\dz{C_\phi} \cap \jd{C_\phi^*} = \{0\}.
   \end{align*}
   \end{cor}
   \begin{proof}
Indeed, if $f \in \dz{C_\phi} \cap \jd{C_\phi^*}$,
then $\|C_\phi f\|= \|C_\phi^* f\|=0$, which means
that $f \in \jd{C_\phi}$. Applying Corollary
\ref{hipnul}, completes the proof.
   \end{proof}
It turns out that composition of $\hsf_\phi$ with
$\phi$ is positive a.e.\ $[\mu]$ (see also the proof
of \cite[Corollary 5]{ha-wh}).
   \begin{pro} \label{hff}
If $\phi\colon X \to X$ is nonsingular, then
$\hsf_\phi \circ \phi > 0$ a.e.\ $[\mu]$.
   \end{pro}
   \begin{proof}
Note that $\mu(\phi^{-1}(\nsf_\phi)) = \int_X
\chi_{\nsf_\phi} \circ \phi \D \mu = \int_X
\chi_{\nsf_\phi} \hsf_\phi \D \mu = 0$. This combined
with $\phi^{-1}(\nsf_\phi) = \{x \in X \colon
\hsf_\phi (\phi (x))=0\}$ completes the proof.
   \end{proof}
   \begin{cor}    \label{hfifi}
If $\phi$ is a nonsingular transformation of $X$ and
$\hsf_\phi \circ \phi = \hsf_\phi$ a.e.\ $[\mu]$, then
$\jd{C_\phi} = \{0\}$.
   \end{cor}
   \begin{proof}
Apply Propositions \ref{nullspace} and \ref{hff}.
   \end{proof}
   \section{\label{ppolar}The polar decomposition}
Given an $\ascr$-measurable function $u\colon X \to
\cbb$, we denote by $M_u$ the operator of
multiplication by $u$ in $L^2(\mu)$ defined by
   \begin{align*}
\dz{M_u} & =\{f \in L^2(\mu)\colon u \cdot f \in
L^2(\mu)\},
   \\
M_u f & = u \cdot f, \quad f \in \dz{M_u}.
   \end{align*}
The operator $M_u$ is a normal operator (cf.\
\cite[Section 7.2]{b-s}.

The polar decomposition of $C_{\phi}$ can be
explicitly described as follows.
   \begin{pro}\label{polar}
Suppose that the composition operator $C_{\phi}$
is densely defined and $C_{\phi}=U|C_{\phi}|$ is
its polar decomposition. Then
   \begin{enumerate}
   \item[(i)] $|C_\phi|=M_{\hsf_\phi^{1/2}}$,
   \item[(ii)] the initial space of $U$ is given
by\,\footnote{\;Note that the mapping $L^2(\hsf_\phi
\D\mu) \ni f \mapsto \hsf_\phi^{1/2} f \in L^2(\mu)$ is
an isometry.}
   \begin{align} \label{initial}
\overline{\ob{|C_{\phi}|}} = \big\{\hsf_\phi^{1/2} f
\colon f \in L^2(\hsf_\phi \D\mu)\big\},
   \end{align}
   \item[(iii)] the final space of $U$ is given by
   \begin{align} \label{final}
\overline{\ob{C_{\phi}}} = \big\{f\circ \phi\colon f
\in L^2(\hsf_\phi \D\mu)\big\},
   \end{align}
   \item[(iv)]  the partial isometry $U$ is given
by\,\footnote{\;Recall that $\hsf_\phi \circ \phi
> 0$ a.e.\ $[\mu]$ (cf.\ Proposition \ref{hff}).}
   \begin{align} \label{upo}
Ug = \frac{g \circ \phi}{(\hsf_\phi \circ
\phi)^{1/2}}, \quad g\in L^2(\mu),
   \end{align}
   \item[(v)]  the adjoint $U^*$ of $U$ is given by
   \begin{align*}
U^* g = \hsf_\phi^{1/2} \cdot V^{-1}Pg, \quad g \in
L^2(\mu),
   \end{align*}
where $V\colon L^2(\hsf_\phi\D\mu) \to
\overline{\ob{C_{\phi}}}$ is a unitary operator
defined by $Vf = f\circ \phi$ for $f\in
L^2(\hsf_\phi\D\mu)$ and $P$ is the orthogonal
projection of $L^2(\mu)$ onto
$\overline{\ob{C_{\phi}}}$.
   \end{enumerate}
   \end{pro}
   \begin{proof}
(i) We will show that $C_{\phi}^* C_{\phi} \subseteq
M_{\hsf_\phi}$. Let $\{X_n\}_{n=1}^\infty$ be as in
Corollary \ref{xn2x} (with $m=1$). Take $f \in
\dz{C_{\phi}^* C_{\phi}}$ and fix $n\in \nbb$. By
\eqref{2}, $\chi_{\varDelta} \in \dz{C_{\phi}}$
whenever $\varDelta \in \ascr$ and $\varDelta
\subseteq X_n$. Thus, for every such $\varDelta$, we
have
   \begin{align*}
\int_{\varDelta} C_{\phi}^* C_{\phi}f \D\mu =
\is{C_{\phi}^* C_{\phi}f}{\chi_{\varDelta}} =
\is{C_{\phi}f}{C_{\phi}\chi_{\varDelta}}
\overset{\eqref{1.5}}= \int_{\varDelta} f \hsf_\phi \D
\mu.
   \end{align*}
Since both functions $(C_{\phi}^*
C_{\phi}f)\chi_{X_n}$ and $(f \hsf_\phi) \chi_{X_n}$
are in $L^1(\mu)$, we deduce that $C_{\phi}^*
C_{\phi}f = f \hsf_\phi$ a.e.\ $[\mu]$ on $X_n$. This
and $X_n\nearrow X$ give $C_{\phi}^* C_{\phi}f = f
\hsf_\phi$ a.e.\ $[\mu]$. As a consequence, we have
$C_{\phi}^* C_{\phi} \subseteq M_{\hsf_\phi}$. Since
both are selfadjoint operators, they are equal. Thus
$|C_{\phi}| = M_{\hsf_\phi}^{1/2} =
M_{\hsf_\phi^{1/2}}$.

(ii) By \cite[Section 8.1]{b-s} and Proposition
\ref{nullspace}, we have
   \begin{align} \label{abc}
\overline{\ob{|C_{\phi}|}} = \jd{|C_{\phi}|}^\perp
=\jd{C_{\phi}}^\perp = \chi_{X \setminus \nsf_{\phi}}
L^2(\mu),
   \end{align}
which as easily seen gives \eqref{initial}.

(iii) By \eqref{base} and (ii), the mapping $W\colon
\overline{\ob{|C_{\phi}|}} \to L^2(\mu)$ given by
   \begin{align}  \label{60knu}
W(\hsf_\phi^{1/2} f) = f \circ \phi, \quad f \in
L^2(\hsf_\phi \D\mu).
   \end{align}
is a well-defined isometry. Using (i) we verify that
$W|_{\ob{|C_{\phi}|}} = U|_{\ob{|C_{\phi}|}}$, which
implies that $\overline{\ob{C_{\phi}}} =
\ob{U}=\ob{W}$. Hence (iii) holds and, by
\eqref{60knu}, we have
   \begin{align} \label{numer}
U^*(f\circ \phi) = \hsf_\phi^{1/2} f, \quad f \in
L^2(\hsf_\phi \D\mu).
   \end{align}

(iv) Applying the measure transport theorem to the
restriction of $\phi$ to the full $\mu$-measure set on
which $\hsf_\phi \circ \phi$ is positive (cf.\
Proposition \ref{hff}), we get
   \begin{align} \label{klucz}
\int_X \frac{|g\circ \phi|^2}{\hsf_\phi \circ \phi} \D
\mu = \int_{X \setminus \nsf_\phi} |g|^2 \D\mu, \quad
g \in L^2(\mu).
   \end{align}
This and Proposition \ref{nullspace} imply that the
mapping $\tilde U \colon L^2(\mu) \ni g \mapsto \frac{g
\circ \phi}{(\hsf_\phi \circ \phi)^{1/2}} \in L^2(\mu)$
is a contraction such that $\jd{\tilde
U}=\chi_{\nsf_\phi}L^2(\mu) = \jd{|C_\phi|}$. Hence, by
\eqref{abc} and \eqref{klucz}, $\tilde U$ is an
isometry on $\overline{\ob{|C_\phi|}}$. Clearly, by
(i), $\tilde U|C_\phi|g = C_\phi g$ for $g \in
\dz{C_\phi}=\dz{|C_\phi|}$, which implies that
$U=\tilde U$.

(v) By \eqref{base} and \eqref{final}, $V$ is a
well-defined unitary operator. If $g \in L^2(\mu)$,
then by (iii), $Pg=f\circ \phi$ a.e.\ $[\mu]$ for some
$f \in L^2(\hsf_\phi\D \mu)$. Thus, by
$\jd{U^*}=\ob{I-P}$ and \eqref{numer}, we have
   \begin{align*}
U^*g = U^*P g = U^*(f\circ \phi) = \hsf_\phi^{1/2} f =
\hsf_\phi^{1/2} \cdot V^{-1} Pg.
   \end{align*}
This completes the proof.
   \end{proof}
Regarding Proposition \ref{polar}, we note that the
formulas for $|C_\phi|$ and $\overline{\ob{C_\phi}}$
are well-known in the case of bounded composition
operators (cf.\ \cite[Lemma 1]{ha-wh}). The formula
\eqref{upo} has appeared in \cite[p.\ 387]{bu-ju-la}
in the context of bounded operators without proof.
   \begin{cor}\label{cos}
Suppose that $C_{\phi}$ is densely defined and $g \in
L^2(\mu)$. Then $g$ belongs to
$\overline{\ob{C_{\phi}}}$ if and only if one of the
following equivalent conditions holds\footnote{\;See
Appendix for definitions and notation.}{\em :}
   \begin{enumerate}
   \item[(i)] there is  an $\ascr$-measurable
function $f\colon X \to \cbb$ such that $g=f\circ
\phi$ a.e.\ $[\mu]$,
   \item[(ii)] there is  a $\phi^{-1}(\ascr)$-measurable
function $f\colon X \to \cbb$ such that $g=f$
a.e.\ ~ $[\mu]$,
   \item[(iii)] $g$ is
$(\phi^{-1}(\ascr))^\mu$-measurable,
   \item[(iv)] for every Borel set $\varDelta$ in $\cbb$
there exists $\varDelta^\prime \in \ascr$ such
that
   \begin{align*}
\mu\big(g^{-1}(\varDelta) \vartriangle
\phi^{-1}(\varDelta^\prime)\big)=0.
   \end{align*}
   \end{enumerate}
In particular, $\overline{\ob{C_{\phi}}} =
L^2(\mu|_{(\phi^{-1}(\ascr))^\mu})$.
   \end{cor}
   \begin{proof}  Apply  \eqref{base}, \eqref{final},
\eqref{complascr}, \eqref{well-k} and Lemma
\ref{hahames-c}.
   \end{proof}
   \begin{cor} \label{adjoint}
If $C_{\phi}$ is densely defined, then the map $V\colon
L^2(\hsf_\phi\D\mu) \to \overline{\ob{C_{\phi}}}$ given
by $Vf = f\circ \phi$ for $f\in L^2(\hsf_\phi\D\mu)$ is
a well-defined unitary operator such that
   \begin{align} \label{adjoint2}
   \begin{aligned} \dz{C_{\phi}^*} & = \big\{g \in
L^2(\mu)\colon \hsf_\phi \cdot V^{-1}Pg \in
L^2(\mu)\big\},
   \\
C_{\phi}^*g & = \hsf_\phi \cdot V^{-1}Pg, \quad g \in
\dz{C_{\phi}^*},
   \end{aligned}
   \end{align}
where $P$ is the orthogonal projection of $L^2(\mu)$
onto
$\overline{\ob{C_{\phi}}}=L^2(\mu|_{(\phi^{-1}(\ascr))^\mu})$.
   \end{cor}
   \begin{proof}
If $C_{\phi}=U|C_{\phi}|$ is the polar decomposition
of $C_{\phi}$, then $C_{\phi}^*=|C_{\phi}|U^*$. This,
Proposition \ref{polar} and Corollary \ref{cos}
complete the proof.
   \end{proof}
   \begin{rem}  \label{remE}
Concerning Corollary \ref{adjoint}, we observe that,
in view of \eqref{CE-1},
$\esf(g):=\esf(g|\phi^{-1}(\ascr))=Pg$ a.e.\ $[\mu]$
and thus $C_\phi^* g = \hsf_\phi \cdot (\esf(g)\circ
\phi^{-1})$ for every $g\in \dz{C_\phi^*}$, where
$\esf(g)\circ \phi^{-1}$ is understood as in
\cite[Lemma 6.4]{ca-hor}.
   \end{rem}
   \section{\label{naqns}Normality and quasinormality}
It turns out that the characterizations of
quasinormality and normality of unbounded composition
operators take the same forms as those for bounded
ones.
   \begin{pro}\label{quasi}
If $C_{\phi}$ is densely defined, then $C_{\phi}$ is
quasinormal if and only if $\hsf_\phi = \hsf_\phi
\circ \phi$ a.e.\ $[\mu]$.
   \end{pro}
   \begin{proof}
Let $C_\phi = U|C_\phi|$ be the polar decomposition of
$C_\phi$. Suppose that $C_\phi$ is quasinormal. Then
by \cite[Proposition 1]{StSz1}, $U |C_\phi| \subseteq
|C_\phi| U$. Let $\{X_n\}_{n=1}^\infty$ be as in
Corollary \ref{xn2x} (with $m=1$). Then, by \eqref{2},
$\{\chi_{X_n}\}_{n=1}^\infty \subseteq \dz{C_\phi}$,
which together with Proposition \ref{polar} implies
that for every $n\in \nbb$,
   \begin{align*}
\chi_{X_n} \circ \phi = U |C_\phi| \chi_{X_n} =
|C_\phi| U \chi_{X_n} =
\Big(\frac{\hsf_\phi}{\hsf_\phi\circ \phi}\Big)^{1/2}
\, \chi_{X_n} \circ \phi \quad \text{a.e.\ $[\mu]$.}
   \end{align*}
Since $X_n \nearrow X$ as $n\to \infty$, we conclude
that $\hsf_\phi = \hsf_\phi \circ \phi$ a.e.\ $[\mu]$.

For the converse, take $f \in \dz{|C_\phi|}$. By
\eqref{upo} and $\dz{|C_\phi|}=\dz{C_\phi}$, we have
   \begin{align*}
\hsf_\phi^{1/2} Uf =
\Big(\frac{\hsf_\phi}{\hsf_\phi\circ \phi}\Big)^{1/2}
f\circ \phi = f\circ \phi \in L^2(\mu).
   \end{align*}
Hence, by Proposition \ref{polar}\,(i), $f \in
\dz{|C_\phi|U}$ and $|C_\phi|U f = C_\phi f =
U|C_\phi|f$. Therefore, $U|C_\phi| \subseteq
|C_\phi|U$. Applying \cite[Proposition 1]{StSz1}
completes the proof.
   \end{proof}
   \begin{pro}
If $\overline{\dz{C_{\phi}}}=L^2(\mu)$, then the
following are equivalent{\em :}
   \begin{enumerate}
   \item[(i)] $C_{\phi}$ is normal,
   \item[(ii)] $\hsf_\phi = \hsf_\phi
\circ \phi$ a.e.\ $[\mu]$ and $\jd{C_{\phi}^*}
\subseteq \jd{C_{\phi}}$,
   \item[(iii)] $\hsf_\phi = \hsf_\phi
\circ \phi$ a.e.\ $[\mu]$ and $\jd{C_{\phi}^*}=\{0\}$,
   \item[(iv)] $\hsf_\phi = \hsf_\phi
\circ \phi$ a.e.\ $[\mu]$ and for every $\varDelta \in
\ascr$ there exists $\varDelta^\prime \in \ascr$ such
that $\mu(\varDelta \vartriangle
\phi^{-1}(\varDelta^\prime))=0$.
   \end{enumerate}
Moreover, if $C_{\phi}$ is normal, then
$\jd{C_{\phi}}=\{0\}$ and $\hsf_\phi > 0$ a.e.\
$[\mu]$.
   \end{pro}
   \begin{proof}
(i)$\Rightarrow$(iii) Since normal operators are
always quasinormal, we infer from Proposition
\ref{quasi} that $\hsf_\phi = \hsf_\phi \circ \phi$
a.e.\ $[\mu]$. Clearly, $\jd{C_{\phi}} =
\jd{C_{\phi}^*}$. That $\jd{C_{\phi}^*}=\{0\}$ follows
from Corollary \ref{hfifi}.

(iii)$\Rightarrow$(ii) Evident.

(ii)$\Rightarrow$(i) This is a direct consequence of
Proposition \ref{quasi} and Theorem \ref{quasi2n}.

(iii)$\Leftrightarrow$(iv) Since $\jd{C_{\phi}^*}
=\{0\}$ if and only if $\ob{C_{\phi}}$ is dense in
$L^2(\mu)$, it suffices to apply Corollary \ref{cos},
Lemma \ref{l2zup} and \eqref{complascr}.

The ''moreover'' part follows from the above and
Proposition \ref{hff}.
   \end{proof}
   \section{\label{fns}Formal normality}
In this section we show that formally normal
composition operators are automatically normal. We
begin by proving a result which is of
measure-theoretical nature. We refer the reader to
Appendix \ref{app2} for the definition and basic
properties of $\esf(\cdot|\phi^{-1}(\ascr))$. For
brevity, we write $\esf(\cdot) =
\esf(\cdot|\phi^{-1}(\ascr))$.
   \begin{lem} \label{hf2}
If $\phi$ is a nonsingular transformation of $X$, then
the following two conditions are equivalent for every
$n\in \nbb${\em :}
   \begin{enumerate}
   \item[(i)] $\hsf_{\phi^{n+1}} = \hsf_{\phi^n} \cdot \hsf_\phi$ a.e.\
$[\mu]$,
   \item[(ii)] $\esf(\hsf_{\phi^n}) = \hsf_{\phi^n}
\circ \phi$ a.e.\ $[\mu|_{\phi^{-1}(\ascr)}]$.
   \end{enumerate}
   \end{lem}
   \begin{proof}
(i)$\Rightarrow$(ii) Note that
   \allowdisplaybreaks
   \begin{align*}
\int_{\phi^{-1}(\varDelta)} \esf(\hsf_{\phi^n}) \D \mu
&= \int_{\phi^{-1}(\varDelta)} \hsf_{\phi^n} \D \mu =
\mu((\phi^{-n}(\phi^{-1}(\varDelta)))
   \\
&= \mu((\phi^{-(n+1)}(\varDelta)) = \int_{\varDelta}
\hsf_{\phi^{n+1}} \D\mu = \int_{\varDelta}
\hsf_{\phi^n} \cdot \hsf_\phi \D\mu
   \\
& = \int_X (\chi_\varDelta \circ \phi) (\hsf_{\phi^n}
\circ \phi) \D\mu = \int_{\phi^{-1}(\varDelta)}
\hsf_{\phi^n} \circ \phi \D\mu, \quad \varDelta \in
\ascr,
   \end{align*}
which, by the uniqueness assertion in the
Radon-Nikodym theorem, implies (ii).

Arguing as above, we can prove the reverse
implication.
   \end{proof}
The next two lemmas are key ingredients of the proof
of Theorem \ref{fn=n}.
   \begin{lem} \label{afn}
Suppose that $C_\phi^2$ is densely defined. Then the
following two conditions are equivalent{\em :}
   \begin{enumerate}
   \item[(i)] $\dz{C_\phi^2} \subseteq
\dz{C_\phi^* C_\phi}$ and $\|C_\phi^2 f\| =
\|C_\phi^*C_\phi f\|$ for all $f\in \dz{C_\phi^2}$,
   \item[(ii)]  $\hsf_{\phi^2} = \hsf_{\phi}^2$ a.e.\
$[\mu]$.
   \end{enumerate}
Moreover, if $C_\phi$ is formally normal, then {\em
(i)} holds.
   \end{lem}
   \begin{proof}
(i)$\Rightarrow$(ii) Take $f \in \dz{C_\phi^2}$. Then,
by Proposition \ref{polar}(i), we have
   \begin{align}  \label{int2}
\int_X |f|^2 \hsf_\phi^2 \D \mu = \|M_{\hsf_\phi}
f\|^2 = \|C_\phi^* C_\phi f\|^2 = \|C_{\phi^2} f\|^2 =
\int_X |f|^2 h_{\phi^2} \D\mu.
   \end{align}
Let $\{X_n\}_{n=1}^\infty$ be as in Corollary
\ref{xn2x} (with $m=2$). Then
$\{\chi_{X_n}\}_{n=1}^\infty \subseteq \dz{C_\phi^2}$
and
   \begin{align*}
\int_{\varDelta} \hsf_\phi^2 \D \mu
\overset{\eqref{int2}}= \int_{\varDelta} h_{\phi^2}
\D\mu < \infty, \quad \varDelta \in \ascr, \,
\varDelta \subseteq X_n, \, n \in \nbb,
   \end{align*}
which implies that $\hsf_\phi^2 = h_{\phi^2}$ a.e.\
$[\mu]$ on $X_n$ for every $n\in \nbb$. Hence (ii)
holds.

(ii)$\Rightarrow$(i) Take $f\in \dz{C_\phi^2}$. Then,
by \eqref{3}, $\int_X |f\hsf_\phi|^2 \D\mu = \int_X
|f|^2 \hsf_{\phi^2} \D\mu < \infty$, which means that
$f \in \dz{M_{\hsf_\phi}} = \dz{C_\phi^*C_\phi}$.
Arguing as in \eqref{int2}, we obtain (i).

The ``moreover'' part is obvious.
   \end{proof}
   \begin{lem} \label{glowny}
If $\phi$ is nonsingular transformation of $X$, then
the following conditions are equivalent{\em :}
   \begin{enumerate}
   \item[(i)] $C_\phi$ is normal,
   \item[(ii)] $C_\phi$ is formally normal
and $\overline{\dz{C_\phi^2}}=L^2(\mu)$.
   \end{enumerate}
   \end{lem}
   \begin{proof}
(i)$\Rightarrow$(ii) Evident (since powers of normal
operators are normal, cf.\ \cite{b-s}).

(ii)$\Rightarrow$(i) First we will show that
   \begin{align} \label{step1}
\dz{C_\phi} \cap \overline{\ob{C_\phi}} =
\dz{C_\phi^*} \cap \overline{\ob{C_\phi}}.
   \end{align}
Indeed, if $g \in \dz{C_\phi^*} \cap
\overline{\ob{C_\phi}}$, then by \eqref{final} there
exists $f \in L^2(\hsf_\phi\D\mu)$ such that $g = f
\circ \phi$ a.e.\ $[\mu]$. It follows from Corollary
\ref{adjoint} that $\hsf_\phi f = \hsf_\phi V^{-1}g
\in L^2(\mu)$. This and Lemma \ref{afn} imply that
   \begin{align*}
\int_X |g\circ \phi|^2 \D\mu = \int_X |f \circ
\phi^2|^2 \D\mu = \int_X |f|^2 h_{\phi^2}\D\mu =
\int_X |fh_{\phi}|^2\D\mu< \infty,
   \end{align*}
which means that $g \in \dz{C_\phi}$. This yields
\eqref{step1}.

Let $P$ be the orthogonal projection of $L^2(\mu)$
onto $\overline{\ob{C_{\phi}}}$. We will prove that
   \begin{align} \label{step2}
P \dz{C_\phi} \subseteq \dz{C_\phi}.
   \end{align}
Indeed, take $f \in \dz{C_\phi}$. Since $(I-P)f \in
\jd{C_\phi^*}$ and $\dz{C_\phi} \subseteq
\dz{C_\phi^*}$, we get $Pf \in \dz{C_\phi^*} \cap
\overline{\ob{C_\phi}}$. Hence by \eqref{step1},
$Pf\in \dz{C_\phi}$, which proves \eqref{step2}.

It follows from \eqref{step2} and Corollary
\ref{fnnul} that
   \begin{align*}
\dz{C_\phi} \subseteq \big(\dz{C_\phi}\cap
\jd{C_\phi^*}\big) \oplus \big(\dz{C_\phi}\cap
\overline{\ob{C_\phi}}\big) = \dz{C_\phi}\cap
\overline{\ob{C_\phi}},
   \end{align*}
which together with $\overline{\dz{C_\phi}}=L^2(\mu)$
imply that $\overline{\ob{C_\phi}}=L^2(\mu)$.
Therefore, by \eqref{step1},
$\dz{C_\phi}=\dz{C_\phi^*}$, which completes the
proof.
   \end{proof}
As is shown below, the assumption
$\overline{\dz{C_\phi^2}}=L^2(\mu)$ in Lemma
\ref{glowny} can be dropped without spoiling its
conclusion.
   \begin{thm} \label{fn=n}
Let $\phi$ be a nonsingular transformation of $X$.
Then $C_\phi$ is normal if and only if $C_\phi$ is
formally normal.
   \end{thm}
   \begin{proof}
It suffices to prove the ``if'' part. Suppose $C_\phi$
is formally normal. Let $\{X_n\}_{n=1}^\infty
\subseteq \ascr$ be as in Corollary \ref{xn2x} (with
$m=1$). Take $\varDelta \in \ascr$. Since
$\{\chi_{X_n\cap \varDelta}\}_{n=1}^\infty \subseteq
\dz{C_\phi}$, we get (see also Remark \ref{remE})
   \allowdisplaybreaks
   \begin{align*}
\int_{X_n\cap \varDelta} \hsf_\phi \D \mu &
\overset{\eqref{base}}= \|C_\phi (\chi_{X_n\cap
\varDelta})\|^2 = \|C_\phi^* (\chi_{X_n\cap
\varDelta})\|^2
   \\
&\overset{\eqref{adjoint2}}= \int_X \hsf_\phi^2 \cdot
|V^{-1}\esf(\chi_{X_n\cap \varDelta})|^2 \D\mu
   \\
&\hspace{1ex} = \int_X (\hsf_\phi \circ
\phi)(\esf(\chi_{X_n\cap \varDelta}))^2 \D\mu, \quad n
\in \nbb.
   \end{align*}
Using \eqref{CE-2} and Lebesgue's monotone convergence
theorem, we obtain
   \begin{align*}
\int_{\varDelta} \hsf_\phi \D \mu = \int_X (\hsf_\phi
\circ \phi)(\esf(\chi_{\varDelta}))^2 \D\mu, \quad
\varDelta \in \ascr,
   \end{align*}
which yields
   \begin{align*}
\int_{\phi^{-1}(\varDelta)} \hsf_\phi \D \mu =
\int_{\phi^{-1}(\varDelta)} \hsf_\phi \circ \phi
\D\mu, \quad \varDelta \in \ascr.
   \end{align*}
This in turn implies that $\esf(\hsf_\phi)=\hsf_\phi
\circ \phi$ a.e.\ $[\mu|_{\phi^{-1}(\ascr)}]$. By
Lemma \ref{hf2}, $\hsf_{\phi^{2}} = \hsf_{\phi}^2$
a.e.\ $[\mu]$. Since $\hsf_\phi < \infty$ a.e.\
$[\mu]$, we see that $\hsf_\phi + \hsf_{\phi^{2}} <
\infty$ a.e.\ $[\mu]$. Using Corollary
\ref{pr1}\,(ii), we get
$\overline{\dz{C_\phi^2}}=L^2(\mu)$. Applying Lemma
\ref{glowny} completes the proof.
   \end{proof}
   \begin{rem}
Using an unpublished result from \cite{StSz-b} (based
on a model for unbounded quasinormal operators), we
can also prove Theorem \ref{fn=n} as follows. Suppose
$C_\phi$ is formally normal. Then, by the polarization
formula, we have
   \allowdisplaybreaks
   \begin{align}   \notag
\int_X f \bar g \hsf_\phi \D \mu & = \is{C_\phi
f}{C_\phi g} = \is{C_\phi^* f}{C_\phi^* g}
   \\   \notag
& \hspace{-1ex}\overset{\eqref{adjoint2}}= \int_X
\hsf_\phi^2 \big(V^{-1}\esf(f)\big)
\big(\overline{V^{-1}\esf(g)}\big) \D\mu
   \\            \label{mordega}
&= \int_X (\hsf_\phi \circ \phi) \, \esf(f) \,
\overline{\esf(g)} \D\mu, \quad f,g\in \dz{C_\phi}.
   \end{align}
By Propositions \ref{clos} and \ref{nullspace}, and
Corollary \ref{hipnul}, we can assume that $0 <
\hsf_\phi(x) < \infty$ for all $x\in X$. Let
$\{X_n\}_{n=1}^\infty \subseteq \ascr$ be as in
Corollary \ref{xn2x} (with $m=1$). Set $Y_n=\{x \in
X_n\colon \hsf_\phi(x) \Ge 1/n\}$ for $n\in \nbb$.
Clearly, $Y_n \nearrow X$ as $n\to \infty$. Take
$\varDelta \in \ascr$. Since
$\{\chi_{Y_n}\}_{n=1}^\infty, \{\hsf_\phi^{-1}\cdot
\chi_{Y_n \cap \varDelta}\}_{n=1}^\infty \subseteq
\dz{C_\phi}$, we can substitute
$f=\hsf_\phi^{-1}\cdot\chi_{Y_n \cap \varDelta}$ and
$g=\chi_{Y_n}$ into \eqref{mordega}. What we get is
   \allowdisplaybreaks
   \begin{align*}
\mu(Y_n \cap \varDelta)&= \int_X (\hsf_\phi \circ
\phi) \, \esf(\hsf_\phi^{-1}\cdot\chi_{Y_n \cap
\varDelta}) \, \esf(\chi_{Y_n}) \D\mu
   \\
&\hspace{-1.1ex}\overset{\eqref{CE-3}}= \int_{Y_n \cap
\varDelta} \frac{\hsf_\phi \circ \phi}{\hsf_\phi} \,
\esf(\chi_{Y_n}) \D\mu, \quad n \in \nbb.
   \end{align*}
Using \eqref{CE-2} and Lebesgue's monotone convergence
theorem, we obtain
   \begin{align*}
\mu(\varDelta) = \int_{\varDelta} \frac{\hsf_\phi
\circ \phi}{\hsf_\phi} \D\mu, \quad \varDelta \in
\ascr,
   \end{align*}
which implies that $\hsf_\phi \circ \phi = \hsf_\phi$
a.e.\ $[\mu]$. By Proposition \ref{quasi}, $C_\phi$ is
quasinormal. Since quasinormal formally normal
operators are normal (cf.\ \cite{StSz-b}), the proof
is complete.
   \end{rem}
   \section{\label{gsmss}Generating Stieltjes moment sequences}
We begin by proving two lemmas which are main tools in
the proof of Theorem \ref{gsms} below.
   \begin{lem} \label{l2}
Suppose $\phi$ is a nonsingular transformation of $X$
and $\{\ee_n\}_{n=1}^\infty$ is a sequence of subsets
of $L^2(\mu)$ satisfying the following three
conditions{\em :}
   \begin{enumerate}
   \item[(i)] $\ee_n$ fulfils \eqref{num3},
$\ee_n \subseteq \dz{C_\phi^n}$ and
$\overline{\ee_n}=L^2(\mu)$ for all $n\in \nbb$,
   \item[(ii)]
$\big[\|C_\phi^{i+j}f\|^2\big]_{i,j=0}^n \Ge 0$ for
all $f\in \ee_{2n}$ and $n\in \nbb$,
   \item[(iii)]
$\big[\|C_\phi^{i+j+1}f\|^2\big]_{i,j=0}^n \Ge 0$ for
all $f\in \ee_{2n+1}$ and $n\in \nbb$.
   \end{enumerate}
Then the following three assertions hold{\em :}
   \begin{enumerate}
   \item[(a)] $\{\hsf_{\phi^n}(x)\}_{n=0}^\infty$ is a
Stieltjes moment sequence for $\mu$-a.e.\ $x \in X$,
   \item[(b)] $C_\phi^n=C_{\phi^n}$ for every
$n\in \nbb$,
   \item[(c)] $\dzn{C_\phi}$ is a core for
$C_\phi^n$ for every $n\in \zbb_+$.
   \end{enumerate}
   \end{lem}
   \begin{proof}
(a) By (i) and Corollary \ref{pr1}\,(ii), there is no
loss of generality in assuming that $0\Le
\hsf_{\phi^n} (x) < \infty$ for all $x \in X$ and
$n\in \zbb_+$. Using \eqref{base}, we obtain
   \begin{align} \label{num6}
\int_X\Big|\sum_{i,j=0}^n \alpha_i\bar\alpha_j
\hsf_{\phi^{i+j}}\Big||f|^2 \D \mu < \infty, \quad f
\in \dz{C_\phi^{2n}}, \, \{\alpha_i\}_{i=0}^n
\subseteq \cbb, \, n\in \zbb_+.
   \end{align}
If $\{\alpha_i\}_{i=0}^n \subseteq \cbb$, then by (i)
and (ii) we have
   \begin{align*}
0 \Le \sum_{i,j=0}^n \|C_\phi^{i+j} f\|^2 \alpha_i
\bar\alpha_j = \int_X \Big(\sum_{i,j=0}^n
\alpha_i\bar\alpha_j \hsf_{\phi^{i+j}}\Big)|f|^2
\D\mu, \quad f \in \ee_{2n}, \, n\in \nbb.
   \end{align*}
Combining (i), \eqref{num6} and Corollary \ref{l1.5c}
(with $\ee=\ee_{2n}$), we see that
   \begin{gather*}
\text{$\sum_{i,j=0}^n \alpha_i\bar\alpha_j
\hsf_{\phi^{i+j}} \Ge 0$ a.e.\ $[\mu]$ for all
$n\in\nbb$ and $\{\alpha_i\}_{i=0}^n \subseteq \cbb$.}
   \end{gather*}
   Let $Q$ be a countable dense subset of $\cbb$. Then
there exists a set $\varDelta_0 \in \ascr$ such that
$\mu(X\setminus \varDelta_0)=0$ and $\sum_{i,j=0}^n
\alpha_i\bar\alpha_j \hsf_{\phi^{i+j}}(x) \Ge 0$ for
all $n\in\nbb$, $\{\alpha_i\}_{i=0}^n \subseteq Q$ and
$x \in \varDelta_0$. As $Q$ is dense in $\cbb$, we
conclude that $[\hsf_{\phi^{i+j}}(x)]_{i,j=0}^n \Ge 0$
for all $n\in \nbb$ and $x \in \varDelta_0$. Using
(iii) and applying a similar reasoning as above, we
infer that there exists a set $\varDelta_1 \in \ascr$
such that $\mu(X\setminus \varDelta_1)=0$ and
$[\hsf_{\phi^{i+j+1}}(x)]_{i,j=0}^n \Ge 0$ for all $n
\in \nbb$ and $x \in \varDelta_1$. Employing
\eqref{Sti} yields (a).

(b) By (a), there exists $\varDelta \in \ascr$ such
that $\mu(X \setminus \varDelta)=0$, $\hsf_{\phi^0}(x)
= 1$ and $\{\hsf_{\phi^n}(x)\}_{n=0}^\infty$ is a
Stieltjes moment sequence for every $x\in \varDelta$.
Hence, for every $x\in \varDelta$ there exists a Borel
probability measure $\mu_x$ on $\rbb_+$ such that
$\hsf_{\phi^n}(x) = \int_{\rbb_+} s^n \D\mu_x(s)$ for
all $n \in \zbb_+$. This yields
   \allowdisplaybreaks
   \begin{align*}
\bigg(\sum_{j=0}^n \hsf_{\phi^j}\bigg)(x) & =
\int_{\rbb_+} \bigg(\sum_{j=0}^n s^j\bigg) \D \mu_x(s)
   \\
& = \int_{[0,1)} \bigg(\sum_{j=0}^n s^j\bigg) \D
\mu_x(s) + \int_{[1,\infty)} \bigg(\sum_{j=0}^n
s^j\bigg) \D \mu_x(s)
   \\
& \Le (n+1) \int_{[0,1)} 1 \D\mu_x(s) + (n+1)
\int_{[1,\infty)} s^n \D\mu_x(s)
   \\
& \Le (n+1) (1 + \hsf_{\phi^n}) (x), \quad x \in
\varDelta, \, n \in \nbb.
   \end{align*}
Hence the domains of $C_{\phi}^n$ and $C_{\phi^n}$
coincide for all $n\in \nbb$. By \eqref{prnst}, this
gives (b).

(c) Apply (i) and Theorem \ref{Mittag}. This completes
the proof.
   \end{proof}
   \begin{lem} \label{l2-pd}
Suppose that $\phi$ is a nonsingular transformation of
$X$ satisfying the following two conditions{\em :}
   \begin{enumerate}
   \item[(i)] $\dz{C_\phi^n}$ is dense in $L^2(\mu)$
for every $n\in \nbb$,
   \item[(ii)]
$\big[\|C_\phi^{i+j}f\|^2\big]_{i,j=0}^n \Ge 0$ for
all $f\in \dz{C_\phi^{2n}}$ and $n\in \nbb$.
   \end{enumerate}
Then the assertions {\em (a)}, {\em (b)} and {\em (c)}
of Lemma {\em \ref{l2}} hold.
   \end{lem}
   \begin{proof}
Set $\ee_n=\dz{C_\phi^n}$ for $n\in \nbb$. According
to \eqref{3}, each $\ee_n$ satisfies \eqref{num3}.
Substituting $C_\phi f$ for $f$ in (ii) implies that
the hypothesis (iii) of Lemma \ref{l2} is satisfied.
Applying Lemma \ref{l2} completes the proof.
   \end{proof}
   \begin{cor}
If $C_\phi$ is subnormal and
$\overline{\dz{C_\phi^n}}=L^2(\mu)$ for all $n \in
\nbb$, then the assertions {\em (a)}, {\em (b)} and
{\em (c)} of Lemma {\em \ref{l2}} hold.
   \end{cor}
   \begin{proof}
Apply Proposition \ref{nth} and Lemma \ref{l2-pd}.
   \end{proof}
The following theorem completely characterizes
composition operators that generate Stieltjes moment
sequences. It should be compared with Lambert's
characterizations of bounded subnormal composition
operators (cf.\ \cite{lam1}).
   \begin{thm} \label{gsms}
If $\phi$ is a nonsingular transformation of $X$, then
the following conditions are equivalent\/{\em :}
   \begin{enumerate}
   \item[(i)] $C_\phi$ generates  Stieltjes moment
sequences,
   \item[(ii)] $\{\hsf_{\phi^n}(x)\}_{n=0}^\infty$ is a
Stieltjes moment sequence for $\mu$-a.e.\ $x \in X$,
   \item[(iii)] $\overline{\dz{C_\phi^k}} = L^2(\mu)$
for all $k\in \nbb$, and
$\{\mu(\phi^{-n}(\varDelta))\}_{n=0}^\infty$ is a
Stieltjes moment sequence for every $\varDelta \in
\ascr$ such that $\mu(\phi^{-k}(\varDelta)) < \infty$
for all $k \in \zbb_+$,
   \item[(iv)] $\hsf_{\phi^n} <
\infty$ a.e.\ $[\mu]$ for all $n\in \nbb$ and $L(p)\Ge
0$ a.e.\ $[\mu]$ whenever $p(t) \Ge 0$ for all $t \in
\rbb_+$, where $L\colon \cbb[t] \to {\EuScript M}$ is
a linear mapping determined by\footnote{\;To make the
definition of $L$ correct we have to modify
$\hsf_{\phi^n}$ so that $0\Le\hsf_{\phi^n}(x)<\infty$
for all $x\in X$ and $n\in \zbb_+$.}
   \begin{align*}
L(t^n) = \hsf_{\phi^n}, \quad n \in \zbb_+;
   \end{align*}
here $\cbb[t]$ is the set of all complex polynomials
in one real variable $t$ and ${\EuScript M}$ is the
set of all $\ascr$-measurable complex functions on
$X$.
   \end{enumerate}
Moreover, if {\em (i)} holds, then
$C_\phi^n=C_{\phi^n}$ and $\dzn{C_\phi}$ is a core for
$C_\phi^n$ for all $n\in \zbb_+$.
   \end{thm}
   \begin{proof}
(i)$\Rightarrow$(ii) Set $\ee_n=\dzn{C_\phi}$ for $n
\in \nbb$. By \eqref{Sti}, \eqref{3} and Lemma
\ref{l2}, we see that the condition (ii) and the
``moreover'' part hold.

(ii)$\Rightarrow$(i) Take $f \in \dzn{C_\phi}$, $n \in
\zbb_+$ and $\{\alpha_i\}_{i=0}^n\subseteq \cbb$.
Then, by \eqref{Sti}, we have
   \begin{align*}
\sum_{i,j=0}^n \alpha_i \overline{\alpha}_j
\|C_{\phi}^{i+j} f\|^2 \overset{\eqref{base}}= \int_X
\bigg(\sum_{i,j=0}^n \alpha_i \overline{\alpha}_j
\hsf_{\phi^{i+j}}(x)\bigg) |f(x)|^2 \D \mu(x) \Ge 0.
   \end{align*}
Applying the above to $C_\phi f$ in place of $f$, we
deduce that the sequences $\{\|C_\phi^k
f\|^2\}_{k=0}^\infty$ and $\{\|C_\phi^{k+1}
f\|^2\}_{k=0}^\infty$ are positive definite.
Therefore, by \eqref{Sti}, $\{\|C_\phi^k
f\|^2\}_{k=0}^\infty$ is a Stieltjes moment sequence.
It follows from Corollary \ref{pr1}(ii) and Theorem
\ref{Mittag} that $\dzn{C_\phi}$ is dense in
$L^2(\mu)$.

(i)$\Rightarrow$(iii) Evident (because $\chi_\varDelta
\in \dzn{C_\phi}$ for every $\varDelta$ as in (iii)).

(iii)$\Rightarrow$(i) By Theorem \ref{Mittag} the set
$\dzn{C_\phi}$ is dense in $L^2(\mu)$. Consider a
simple $\ascr$-measurable function $u =\sum_{i=1}^k
\alpha_i \chi_{\varDelta_i}$, where
$\{\alpha_i\}_{i=1}^k$ are positive real numbers and
$\{\varDelta_i\}_{i=1}^k$ are pairwise disjoint sets
in $\ascr$. Suppose that $u$ is in $\dzn{C_\phi}$.
Then, by the measure transport theorem,
$\{\chi_{\varDelta_i}\}_{i=1}^k \subseteq
\dzn{C_\phi}$ and
   \begin{align*}
\|C_\phi^n u\|^2 = \sum_{i,j=1}^k \alpha_i \alpha_j
\int_{\varDelta_i \cap \varDelta_j} \hsf_{\phi^n}
\D\mu = \sum_{i=1}^k \alpha_i^2 \int_{\varDelta_i}
\hsf_{\phi^n} \D\mu \overset{\eqref{base}}=
\sum_{i=1}^k \alpha_i^2 \mu(\phi^{-n}(\varDelta_i))
   \end{align*}
for all $n\in \zbb_+$. Hence, by (iii), we have
   \begin{align} \label{uSti}
   \begin{minipage}{65ex}
$\{\|C_\phi^n u\|^2\}_{n=0}^\infty$ is a Stieltjes
moment sequence for every simple nonnegative
$\ascr$-measurable function $u\in\dzn{C_\phi}$.
   \end{minipage}
   \end{align}
Now take $f \in \dzn{C_\phi}$. Then there exists a
sequence $\{u_n\}_{n=1}^\infty$ of simple
$\ascr$-measurable functions $u_n\colon X \to \rbb_+$
such that $u_n(x) \Le u_{n+1}(x) \Le |f(x)|$ and
$\lim_{k \to \infty} u_k(x) = |f(x)|$ for all $n\in
\nbb$ and $x \in X$. This implies that
$\{u_n\}_{n=1}^\infty \subseteq \dzn{C_\phi}$ and, by
Lebesgue's monotone convergence theorem,
   \begin{align*}
\|C_\phi^n f\|^2 = \int_X |f|^2 \hsf_{\phi^n} \D\mu =
\lim_{k\to \infty} \int_X u_k^2 \hsf_{\phi^n} \D\mu =
\lim_{k\to \infty} \|C_\phi^n u_k\|^2, \quad n \in
\zbb_+.
   \end{align*}
Since the class of Stieltjes moment sequences is
closed under the operation of taking pointwise limits
(cf.\ \eqref{Sti}), we infer from \eqref{uSti} that
$\{\|C_\phi^n f\|^2\}_{n=0}^\infty$ is a Stieltjes
moment sequence.

(ii)$\Rightarrow$(iv) If $p\in \cbb[t]$ is such that
$p(t) \Ge 0$ for all $t \in \rbb_+$, then there exist
$q_1,q_2\in\cbb[t]$ such that $p(t)=t|q_1(t)|^2 +
|q_2(t)|^2$ for all $t \in \rbb$ (see \cite[Problem
45, p.\ 78]{P-G}). This fact combined with \eqref{Sti}
implies that $L(p) \Ge 0$ a.e.\ $[\mu]$.

(iv)$\Rightarrow$(ii) Let $Q$ be a countable dense
subset of $\cbb$. If $q\in\cbb[t]$ is a polynomial
with coefficients in $Q$, then the polynomials
$p_1:=|q|^2$ and $p_2:=t|q|^2$ are nonnegative on
$\rbb_+$. Hence $L(p_i) \Ge 0$ a.e.\ $[\mu]$ for
$i=1,2$. Since $Q$ is countable, this implies that
there exists $\varDelta \in \ascr$ such that
$\mu(X\setminus \varDelta)=0$,
   \begin{align} \label{haha}
0 \leq \hsf_{\phi^n}(x)<\infty, \sum_{i,j=0}^n
\alpha_i \overline{\alpha}_j \hsf_{\phi^{i+j}}(x) \Ge
0 \text{ and } \sum_{i,j=0}^n \alpha_i
\overline{\alpha}_j \hsf_{\phi^{i+j+1}}(x) \Ge 0
   \end{align}
for all $n \in \zbb_+$, $\{\alpha_i\}_{i=0}^n
\subseteq Q$ and $x \in \varDelta$. As $Q$ is dense in
$\cbb$, we see that \eqref{haha} holds for all $n \in
\zbb_+$, $\{\alpha_i\}_{i=0}^n \subseteq \cbb$ and $x
\in \varDelta$. This and \eqref{Sti} complete the
proof.
   \end{proof}
   \begin{outc} \label{conclud}
We close the paper by pointing out that there exists a
composition operator generating Stieltjes moment
sequences which is not subnormal and even not
hyponormal. Such an operator can be constructed on the
basis of a weighted shift on a directed tree with one
branching vertex (cf.\ \cite[Section 4.3]{j-j-s0}). In
view of Theorem \ref{gsms}, any composition operator
$C_\phi$ which generates Stieltjes moment sequences,
in particular the aforementioned, satisfies the
conditions (ii), (iii) and (iv) of this theorem as
well as its ``moreover'' part (specifically,
$\dzn{C_\phi}$ is a core for $C_\phi^n$ for every
$n\in \zbb_+$, which is considerably more than is
required in Definition \ref{maindef}). Therefore, none
of the Lambert characterizations of subnormality of
bounded composition operators (cf.\ \cite{lam1}) is
valid in the unbounded case. It is worth mentioning
that the above example is built over the discrete
measure space. However, it can be immediately adapted
to the context of measures which are equivalent to the
Lebesgue measure on $[0,\infty)$ by applying
\cite[Theorem 2.7]{jab}.
   \end{outc}
   \appendix
   \section{}  \label{apms}
Here we gather some useful properties of $L^2$-spaces.
The first two lemmas seem to be folklore. For the
reader's convenience, we include their proofs.
   \begin{lem} \label{l1.11}
Let $(X, \ascr, \mu)$ be a measure space and let
$\rho_1, \rho_2$ be $\ascr$-measurable scalar
functions on $X$ such that $0 < \rho_i < \infty$ a.e.\
$[\mu]$ for $i=1,2$. Then $L^2(\rho_1\D\mu) \cap
L^2(\rho_2\D\mu)$ is dense\footnote{\;This makes sense
because the measures $\rho_1\D\mu$ and $\rho_2\D\mu$
are mutually absolutely continuous.} in $L^2(\rho_i
\D\mu)$ for $i=1,2$.
   \end{lem}
   \begin{proof}
Since $L^2(\rho_1\D\mu) \cap L^2(\rho_2\D\mu) =
L^2((\rho_1+\rho_2)\D\mu)$, we can assume that $0 <
\rho_2(x) \Le \rho_1(x) < \infty$ for all $x \in X$.
Take $\varDelta \in \ascr$ such that $\chi_\varDelta
\in L^2(\rho_2\D\mu)$. Set $\varDelta_n=\big\{x \in
\varDelta\colon \rho_1(x) \Le n \text{ and }
\frac{1}{n} \Le \rho_2(x)\big\}$ for $n\in \nbb$. Note
that $\{\chi_{\varDelta_n}\}_{n=1}^\infty \subseteq
L^2(\rho_1\D\mu)$. Since $\varDelta_n \nearrow
\varDelta$ as $n\to \infty$, we see that
$\{\chi_{\varDelta_n}\}_{n=1}^\infty$ converges to
$\chi_\varDelta$ in $L^2(\rho_2 \D\mu)$. Applying
\cite[Theorem 3.13]{Rud} completes the proof.
   \end{proof}
Note that Lemma \ref{l1.11} is no longer true if one
of the density functions $\rho_1$ and $\rho_2$ takes
the value $\infty$ on a set of positive measure $\mu$
(even if $\rho_2 \Le \rho_1$). Employing Lemma
\ref{l1.11} and the Radon-Nikodym theorem, we get the
following.
   \begin{cor} \label{zlem}
Let $(X, \ascr, \mu_1)$ and $(X, \ascr, \mu_2)$ be
$\sigma$-finite measure spaces. If the measures
$\mu_1$ and $\mu_2$ are mutually absolutely
continuous, then $L^2(\mu_1) \cap L^2(\mu_2)$ is dense
in $L^2(\mu_i)$ for $i=1,2$.
   \end{cor}
Corollary \ref{zlem} is no longer true if one of the
measures $\mu_1$ and $\mu_2$ is not $\sigma$-finite.
   \begin{lem}   \label{Landau}
Let $(X,\ascr,\mu)$ be a $\sigma$-finite measure space
and $\rho_1, \rho_2$ be $\ascr$-measur\-able scalar
functions on $X$ such that $0 < \rho_1 \Le \infty$
a.e.\ $[\mu]$ and $0 \Le \rho_2 \Le \infty$ a.e.\
$[\mu]$. Then the following two conditions are
equivalent\/{\em :}
   \begin{enumerate}
   \item[(i)] $\int_X|f|^2 \rho_2\D\mu < \infty$ for
every $\ascr$-measurable function $f\colon X\to \cbb$
such that $\int_X|f|^2 \rho_1\D\mu < \infty$,
   \item[(ii)]
there exists $c \in \rbb_+$ such that $\rho_2 \Le c
\rho_1$ a.e.\ $[\mu]$.
   \end{enumerate}
   \end{lem}
   \begin{proof}
(i)$\Rightarrow$(ii) Without loss of generality we can
assume that $\rho_1 < \infty$ a.e.\ $[\mu]$. We can
also assume that $\rho_2 < \infty$ a.e.\ $[\mu]$
(indeed, otherwise, since $\rho_1 < \infty$ a.e.\
$[\mu]$ and $\mu$ is $\sigma$-finite, there exist
$\varOmega \in \ascr$ and $k \in \nbb$ such that
$\rho_1(x)\Le k$ and $\rho_2(x)=\infty$ for all $x\in
\varOmega$, and $0<\mu(\varOmega)<\infty$; hence
$\int_{\varOmega} \rho_1 \D\mu < \infty$ and
$\int_{\varOmega} \rho_2 \D\mu = \infty$, which is a
contradiction). Finally, replacing $\rho_2$ by
$\frac{\rho_2}{\rho_1}$ if necessary, we can assume
that $\rho_1(x) = 1$ for all $x\in X$. Now applying
the Landau-Riesz summability theorem (cf.\
\cite[Problem G, p.\ 398]{b-p}), we obtain (ii). The
implication (ii)$\Rightarrow$(i) is obvious.
   \end{proof}
   \begin{cor} \label{l1.12}
Let $(X,\ascr,\mu)$ be a $\sigma$-finite measure space
and $\rho_1, \rho_2$ be $\ascr$-measurable scalar
functions on $X$ such that $0 < \rho_i \Le \infty$
a.e.\ $[\mu]$ for $i=1,2$. Then $L^2(\rho_1\D\mu)
\subseteq L^2(\rho_2\D\mu)$ if and only if there
exists $c \in \rbb_+$ such that $\rho_2 \Le c \rho_1$
a.e.\ ~ $[\mu]$.
   \end{cor}
The implication (i)$\Rightarrow$(ii) of Lemma
\ref{Landau} is not true if we drop the assumption
that $\rho_1 > 0$ a.e.\ $[\mu]$. Corollary \ref{l1.12}
is no longer true if $\mu$ is not $\sigma$-finite
(e.g., $X=\nbb$, $\ascr=2^X$, $\mu(\{1\})=1$,
$\mu(\{i\}) = \infty$ for $i \Ge 2$, $\rho_1\equiv 1$
and $\rho_2(n)=n$ for $n\in X$).

The following lemma generalizes \cite[Lemma 2.1]{jab}.
   \begin{lem} \label{l1.5+}
Let $(X,\ascr,\mu)$ be a $\sigma$-finite measure
space, $\ee$ be a dense subset of $L^2(\mu)$ and $h
\colon X \to \cbb$ be an $\ascr$-measurable function
such that
   \begin{align} \label{num5+}
\text{$\int_\varDelta |h||f|^2\D \mu < \infty$ and
$\int_\varDelta h|f|^2\D \mu \Ge 0$ for all $f \in
\ee$ and $\varDelta \in \ascr_*$,}
   \end{align}
where $\ascr_* = \{\varDelta \in \ascr \colon
\mu(\varDelta) < \infty\}$. Then $h \Ge 0$ a.e.\
$[\mu]$.
   \end{lem}
   \begin{proof}
Set $\varXi_f = \{x \in X \colon |f(x)|
> 0\}$ for $f\in \ee$. First, we will show that
   \begin{align} \label{f1}
\text{$h(x) \Ge 0$ for $\mu$-a.e.\ $x \in \varXi_{f}$
and for every $f\in \ee$.}
   \end{align}
Indeed, fix $f\in \ee$ and set
$\varXi_{f,k}=\big\{x\in X\colon |f(x)| \Ge
\frac{1}{k}\big\}$ for $k \in \nbb$. It follows from
Chebyshev's inequality that $\varXi_{f,k} \in \ascr_*$
for $k\in \nbb$. Applying \eqref{num5+}, we deduce
that
   \begin{align*}
\int_{\varXi_{f,k}} |h||f|^2\D \mu < \infty \text{ and
} \int_{\varXi_{f,k} \cap \varDelta} h|f|^2\D \mu \Ge
0 \text{ for all } \varDelta\in \ascr \text{ and } k
\in \nbb.
   \end{align*}
This implies that $h \Ge 0$ a.e.\ $[\mu]$ on
$\varXi_{f,k}$ for every $k \in \nbb$. Since
$\varXi_{f,k} \nearrow \varXi_{f}$ as $k\to\infty$, we
conclude that $h \Ge 0$ a.e.\ $[\mu]$ on $\varXi_{f}$.

Set $\varSigma=\{x\in X \colon h(x) \Ge 0\}$. Suppose
that, contrary to our claim, $\mu(X \setminus
\varSigma) > 0$. As $\mu$ is $\sigma$-finite, there
exists a set $\varOmega \in \ascr$ such that
$\varOmega \subseteq X \setminus \varSigma$ and $0 <
\mu(\varOmega) < \infty$. This means that
$\chi_\varOmega \in L^2(\mu)$. Since $\ee$ is dense in
$L^2(\mu)$, there exists a sequence
$\{f_n\}_{n=1}^\infty \subseteq \ee$ which converges
to $\chi_\varOmega$ in $L^2(\mu)$. Passing to a
subsequence if necessary, we can assume that the
sequence $\{f_n\}_{n=1}^\infty$ converges a.e.\
$[\mu]$ to $\chi_\varOmega$, and thus
   \begin{align} \label{f3}
\text{$\lim_{n\to\infty} f_n(x) = \chi_\varOmega(x)=1$
for $\mu$-a.e.\ $x \in \varOmega$.}
   \end{align}
It follows from \eqref{f1} that $\mu(\varOmega \cap
\varXi_f) = 0$ for every $f\in\ee$. Applying this
property to $f=f_n$ ($n\in \nbb$), we see that
$\mu\Big(\varOmega \cap \bigcup_{n=1}^\infty
\varXi_{f_n}\big) = 0$, which means that
   \begin{align} \label{f2}
\text{$f_n(x)=0$ for all $n \in \nbb$ and for
$\mu$-a.e.\ $x \in \varOmega$.}
   \end{align}
Combining \eqref{f3} with \eqref{f2}, we conclude that
$\mu(\varOmega)=0$, a contradiction.
   \end{proof}
Applying Lemma \ref{l1.5+} to $h$ and $-h$, we see
that this lemma remains valid if ``$\Ge$'' is replaced
by ``$=$''.
   \begin{cor} \label{l1.5c}
Let $(X,\ascr,\mu)$ be a $\sigma$-finite measure space
and $\ee$ be a dense subset of $L^2(\mu)$ such that
   \begin{align} \label{num3}
\text{$f\chi_\varDelta \in \ee$ for all $f \in \ee$
and $\varDelta \in \ascr_*$.}
   \end{align}
If $h \colon X \to \cbb$ is an $\ascr$-measurable
function such that $\int_X |h||f|^2\D \mu < \infty$
and $\int_X h|f|^2\D \mu \Ge 0$ for all $f \in \ee$,
then $h \Ge 0$ a.e.\ $[\mu]$.
   \end{cor}
   \section{} \label{app2}
In this appendix, we describe (mostly without proofs)
some results from measure theory which play an
important role in our analysis of composition
operators. Let $(X, \ascr, \mu)$ be a fixed measure
space and let $\bscr\subseteq \ascr$ be a
$\sigma$-algebra. We say that $\bscr$ is {\em
relatively $\mu$-complete} if $\ascr_0 \subseteq
\bscr$, where $\ascr_0= \{\varDelta \in \ascr\colon
\mu(\varDelta)=0\}$ (cf.\ \cite{ha-wh}). It is easily
seen that the smallest relatively $\mu$-complete
$\sigma$-algebra containing $\bscr$, denoted by
$\bscr^\mu$, coincides with the $\sigma$-algebra
generated by $\bscr \cup \ascr_0$, and that
   \begin{align} \label{complascr}
\bscr^\mu = \{\varDelta \in \ascr \, | \, \exists
\varDelta^\prime \in \bscr \colon \mu(\varDelta
\vartriangle \varDelta^\prime)=0\}.
   \end{align}
The $\bscr^\mu$-measurable functions are described
below (cf.\ \cite[Lemma 1, p.\ 169]{Rud}).
   \begin{lem}\label{hahames}
A function $f\colon X \to \cbb$ is
$\bscr^\mu$-measurable if and only if there exists a
$\bscr$-measurable function $g\colon X \to \cbb$ such
that $f=g$ a.e.\ $[\mu]$.
   \end{lem}
By the above lemma $L^2(\mu|_{\bscr})$ is a subset of
$L^2(\mu)$ if and only if $\bscr=\bscr^{\mu}$. The
question of when $L^2(\mu|_{\bscr})=L^2(\mu)$ has a
simple answer ($\sigma$-finiteness is essential!).
   \begin{lem} \label{l2zup}
If $\mu$ is $\sigma$-finite and $\bscr$ is relatively
$\mu$-complete, then $L^2(\mu|_{\bscr})=L^2(\mu)$ if
and only if $\bscr=\ascr$.
   \end{lem}
   \begin{proof}
Suppose that $L^2(\mu|_{\bscr})=L^2(\mu)$ and
$\varDelta \in \ascr \setminus \bscr$. Since $\mu$ is
$\sigma$-finite, we may assume that
$\mu(\varDelta)<\infty$. Then $\chi_{\varDelta} \in
L^2(\mu)\setminus L^2(\mu|_{\bscr})$, a contradiction.
   \end{proof}
Given a transformation $\phi$ of $X$, we set
$\phi^{-1}(\ascr) = \{\phi^{-1}(\varDelta)\colon
\varDelta \in \ascr\}$.
   \begin{lem} \label{hahames-c}
Suppose that $\phi\colon X \to X$ is an
$\ascr$-measurable transformation and $f\colon X \to
\cbb$ is an arbitrary function. Then $f$ is
$(\phi^{-1}(\ascr))^\mu$-measurable if and only if
there exists an $\ascr$-measurable function $u\colon X
\to \cbb$ such that $f = u \circ \phi$ a.e\ $[\mu]$.
   \end{lem}
   \begin{proof}
Applying the following well-known fact
   \begin{align} \label{well-k}
   \begin{minipage}{65ex}
a function $g\colon X \to \cbb$ is
$\phi^{-1}(\ascr)$-measurable if and only if there
exists an $\ascr$-measurable function $u\colon X \to
\cbb$ such that $g=u \circ \phi$,
   \end{minipage}
   \end{align}
and Lemma \ref{hahames} completes the proof.
   \end{proof}
Let $P_{\bscr}$ be the orthogonal projection of
$L^2(\mu)$ onto its closed subspace
$L^2(\mu|_{\bscr^\mu})$. Set $\bscr_{*}=\{\varDelta
\in \bscr\colon \mu(\varDelta)<\infty\}$. It follows
from Lemma \ref{hahames} that
   \begin{align} \label{CE-1}
   \begin{minipage}{75ex}
for every $f\in L^2(\mu)$ there exists a unique (up to
sets of measure zero) $\bscr$-measurable function
$\esf(f|\bscr)\colon X \to \cbb$ such that $P_\bscr f
=\esf(f|\bscr)$ a.e.\ $[\mu]$.
   \end{minipage}
   \end{align}
This and the fact that $\is{\chi_\varDelta}{f} =
\is{\chi_\varDelta}{P_\bscr f}$ for all $f\in
L^2(\mu)$ and $\varDelta \in \bscr_{*}$ yield
   \begin{align} \label{przegad}
\int_\varDelta f \D\mu = \int_\varDelta \esf(f|\bscr)
\D\mu, \quad f\in L^2(\mu), \, \varDelta \in
\bscr_{*}.
   \end{align}
Now suppose that $\mu|_{\bscr}$ is $\sigma$-finite. It
follows from \eqref{przegad} that $\esf(f|\bscr) \Ge
0$ a.e.\ $[\mu]$ whenever $f\Ge 0$ a.e.\ $[\mu]$. By
applying the standard approximation procedure, we see
that for every $\ascr$-measurable function $f\colon X
\to [0,\infty]$ there exists a unique (up to sets of
measure zero) $\bscr$-measurable function
$\esf(f|\bscr)\colon X \to [0,\infty]$ such that the
equality in \eqref{przegad} holds for every $\varDelta
\in \bscr$. Thus for every $\ascr$-measurable function
$f\colon X \to [0,\infty]$ and for every
$\bscr$-measurable function $g\colon X \to [0,\infty]$
we have
   \begin{align} \label{CE-3}
\int_X g f \D\mu= \int_X g \esf(f|\bscr) \D\mu.
   \end{align}
We call $\esf(f|\bscr)$ the {\em conditional
expectation} of $f$ with respect to $\bscr$ (cf.\
\cite{Rao}). Clearly,
   \begin{align} \label{CE-2}
   \begin{minipage}{70ex}
if $0\Le f_n \nearrow f$ are $\ascr$-measurable, then
$\esf(f_n|\bscr) \nearrow \esf(f|\bscr)$,
   \end{minipage}
   \end{align}
where $g_n \nearrow g$ means that for $\mu$-a.e.\
$x\in X$, the sequence $\{g_n(x)\}_{n=1}^\infty$ is
monotonically increasing and convergent to $g(x)$.

Concluding Appendix \ref{app2}, we note that if $\mu$
is $\sigma$-finite and $\phi\colon X \to X$ is a
nonsingular transformation such that $\hsf_\phi <
\infty$ a.e.\ $[\mu]$ (equivalently, $C_\phi$ is
densely defined), then the measure
$\mu|_{\phi^{-1}(\ascr)}$ is $\sigma$-finite (cf.\
Proposition \ref{clos}). Thus we may consider the
conditional expectation $\esf(\cdot|\phi^{-1}(\ascr))$
with respect to $\phi^{-1}(\ascr)$.

   \bibliographystyle{amsalpha}

\begin{thebibliography}{99}
   \bibitem{ber} C. Berg, J. P. R.
Christensen, P. Ressel, {\em Harmonic Analysis on
Semigroups}, Springer, Berlin, 1984.
   \bibitem{B-C} G. Biriuk, E. A. Coddington,
Normal extensions of unbounded formally normal
operators, {\em J. Math. Mech.} {\bf 12} (1964),
617-638.
   \bibitem{b-s} M. Sh. Birman, M. Z. Solomjak,
{\it Spectral theory of selfadjoint operators in
Hilbert space}, D. Reidel Publishing Co., Dordrecht,
1987.
   \bibitem{bro} A. Brown,  On a class of operators, {\em
Proc. Amer. Math. Soc.} {\bf 4}, (1953), 723-728.
   \bibitem{b-p} A. Brown, C. Pearcy, {\em Introduction
to operator theory. I. Elements of functional
analysis}, Graduate Texts in Mathematics, No. 55.
Springer-Verlag, New York-Heidelberg, 1977.
   \bibitem{budz} P.\ Budzy\'{n}ski, A note on
unbounded hyponormal composition operators in
$L^2$-spaces, submitted.
   \bibitem{b-j-j-s} P. Budzy\'{n}ski,
Z. J. Jab{\l}o\'nski, I. B. Jung, J. Stochel,
Unbounded subnormal weighted shifts on directed trees,
submitted.
   \bibitem{B-S1} P. Budzy\'{n}ski, J. Stochel, Joint
subnormality of $n$-tuples and $C_0$-semigroups of
composition operators on $L^2$-spaces, {\em Studia
Math.} {\bf 179} (2007), 167-184.
   \bibitem{B-S2} P. Budzy\'{n}ski, J. Stochel, Joint
subnormality of $n$-tuples and $C_0$-semigroups of
composition operators on $L^2$-spaces, II, {\em Studia
Math.} {\bf 193} (2009), 29-52.
   \bibitem{bu-ju-la} Ch. Burnap, I. B. Jung,
A. Lambert, Separating partial normality classes with
composition operators, {\em J. Operator Theory} {\bf
53} (2005), 381-397.
   \bibitem{ca-hor} J. T. Campbell, W. E. Hornor,
Seminormal composition operators. {\em J. Operator
Theory} {\bf 29} (1993), 323-343.
   \bibitem{Cod-1} E. A. Coddington,
Normal extensions of formally normal operators, {\em
Pacific J. Math.} {\bf 10} (1960), 1203-1209.
   \bibitem{Cod} E. A. Coddington, Formally normal operators
having no normal extension, {\em Canad. J. Math.} {\bf
17} (1965), 1030-1040.
   \bibitem{con} J. B. Conway, {\em The theory of
subnormal operators}, Mathematical Surveys and
Monographs, Providence, Rhode Island, 1991.
   \bibitem{da-st} A. Daniluk, J. Stochel, Seminormal
composition operators induced by affine
transformations, {\it Hokkaido Math. J.} {\bf 26}
(1997), 377-404.
   \bibitem{di-ca} P. Dibrell, J. T. Campbell, Hyponormal powers
of composition operators, {\em Proc. Amer. Math. Soc.}
{\bf 102} (1988), 914-918.
   \bibitem{dun-sch} N. Dunford, J. T. Schwartz, {\em Linear
operators,} Part I, Interscience, New York-London
1958.
   \bibitem{emb-lam2}  M. R. Embry, A. Lambert,
Subnormal weighted translation semigroups, {\em J.
Funct. Anal.} {\bf 24} (1977), 268-275.
   \bibitem{emb-lam4} M. R. Embry-Wardrop, A. Lambert,
Measurable transformations and centered composition
operators, {\em Proc. Royal Irish Acad.,} {\bf 90A}
(1990), 165-172.
   \bibitem{emb-lam3}  M. R. Embry-Wardrop, A. Lambert,
Subnormality for the adjoint of a composition operator
on $L^2$, {\em J. Operator Theory} {\bf 25} (1991),
309-318.
   \bibitem{hal1} P. R. Halmos, Normal dilations and extensions of
operators, {\em Summa Brasil. Math.} {\bf 2} (1950),
125-134.
   \bibitem{hal2} P. R. Halmos, {\em Measure theory}, van
Nostrand, Princeton 1956.
   \bibitem{ha-wh} D. Harrington, R. Whitley, Seminormal
composition operators, {\em J. Operator Theory}, {\bf
11} (1984), 125-135.
   \bibitem{jab} Z. J. Jab{\l}onski, Hyperexpansive composition operators,
{\em Math. Proc. Cambridge Philos. Soc.} {\bf 135}
(2003), 513-526.
   \bibitem{j-j-s0} Z. J. Jab{\l}o\'nski, I. B. Jung,
J. Stochel, A non-hyponormal operator generating
Stieltjes moment sequences, to appear in {\em J.
Funct. Anal}.
   \bibitem{j-j-s3} Z. J. Jab{\l}o\'nski,  I. B. Jung,
J. Stochel, A hyponormal weighted shift on a directed
tree whose square has trivial domain, submitted.
   \bibitem{jj3} J. Janas, On unbounded hyponormal
operators. III, {\em Studia Math.} {\bf 112} (1994),
75-82.
   \bibitem{Ko-Th} S. Kouchekian,  J. E. Thomson,
A note on density for the core of unbounded Bergman
operators, {\em Proc. Amer. Math. Soc.} {\bf 139}
(2011), 2067-2072.
   \bibitem{lam} A. Lambert, Subnormality and weighted
shifts, {\em J. London Math. Soc.} {\bf 14} (1976),
476-480.
   \bibitem{lam3} A. Lambert, Hyponormal composition operators,
{\em Bull. London Math. Soc.} {\bf 18} (1986),
395-400.
   \bibitem{lam1} A. Lambert, Subnormal composition operators,
{\em Proc. Amer. Math. Soc.} {\bf 103} (1988),
750-754.
   \bibitem{lam2} A. Lambert, Normal extensions of subnormal
composition operators, {\em Michigan Math. J.} {\bf
35} (1988), 443-450.
   \bibitem{ml} W. Mlak, Operators induced by transformations of
Gaussian variables, {\em Ann. Polon. Math.} {\bf 46}
(1985), 197-212.
   \bibitem{nor} E. Nordgren, Composition operators on Hilbert
spaces, Lecture Notes in Math. {\bf 693},
Springer-Verlag, Berlin 1978, 37-63.
   \bibitem{ot-sch} S. \^Ota, K.  Schm\"udgen,   On
some classes of unbounded operators, {\em Integr.
Equat. Oper. Th.} {\bf 12} (1989), 211-226.
   \bibitem{P-G} G. P\'olya,  G. Szeg\"o,
{\em Problems and theorems in analysis, Vol. II,
Theory of functions, zeros, polynomials, determinants,
number theory, geometry}, revised and enlarged
translation by C. E. Billigheimer of the fourth German
edition, Die Grundlehren der Mathematischen
Wissenschaften, Band 216. Springer-Verlag, New
York-Heidelberg, 1976.
   \bibitem{Rao} M. M. Rao, {\em Conditional measures
and applications}, Pure and Applied Mathematics (Boca
Raton) 271, Chapman and Hall/CRC, Boca Raton, 2005.
   \bibitem{Rud} W. Rudin, {\em Real and Complex
Analysis}, McGraw-Hill, New York 1987.
   \bibitem{Sch0} K.\,Schm\"{u}dgen,
On domains of powers of closed symmetric operators,
{\em J. Operator Theory} {\bf 9} (1983), 53-75.
   \bibitem{Sch1} K.\,Schm\"{u}dgen,  A   formally   normal
operator having no normal extension, {\em Proc. Amer.
Math. Soc.} {\bf 95} (1985), 503-504.
   \bibitem{Sch} K. Schm\"{u}dgen, {\em Unbounded operator
algebras and representation theory}, Operator Theory:
Advances and Applications, {\bf 37}, Birkh\"{a}user
Verlag, Basel, 1990.
   \bibitem{sin} R. K. Singh,  Compact and quasinormal
composition operators, {\em Proc. Amer. Math. Soc.,}
{\bf 45} (1974), 80-82.
   \bibitem{sin-man} R. K. Singh, J. S. Manhas,
{\em Composition operators on function spaces},
Elsevier Science Publishers B.V.,
North-Holland-Amsterdam, 1993.
   \bibitem{sin-vel} R. K. Singh, T. Veluchamy,
Spectrum of a normal composition operators on $L^2$
spaces, {\em Indian J. Pure Appl. Math.,} {\bf 16}
(1985), 1123-1131.
   \bibitem{sto} J. Stochel, Seminormal composition
operators on $L^2$ spaces induced by matrices, {\em
Hokkaido Math. J.} {\bf 19} (1990), 307-324.
   \bibitem{sto1} J. Stochel, Moment functions on
real algebraic sets, {\em Ark. Mat.} {\bf 30} (1992),
133-148.
   \bibitem{2xSt} J. Stochel, J. B. Stochel, Seminormal
composition operators on $L^2$ spaces induced by
matrices:\ the Laplace density case, {\em J. Math.
Anal. Appl.} {\bf 375} (2011), 1-7.
    \bibitem{StSz1} J. Stochel, F. H. Szafraniec,
On normal extensions of unbounded operators. I, {\it
J. Operator Theory} {\bf 14} (1985), 31-55.
    \bibitem{StSz2} J. Stochel and F. H. Szafraniec, On normal
extensions of unbounded operators, II, {\em Acta Sci.
Math. $($Szeged$)$} {\bf 53} (1989), 153-177.
   \bibitem{StSz3} J. Stochel, F. H. Szafraniec,
On normal extensions of unbounded operators. III.
Spectral properties, {\em Publ. RIMS, Kyoto Univ.}
{\bf 25} (1989), 105-139.
   \bibitem{StSz} J. Stochel, F. H. Szafraniec, The
complex moment problem and subnormality: a polar
decomposition approach, {\em J. Funct. Anal.}
159(1998), 432-491.
   \bibitem{StSz-b} J. Stochel, F. H.
Szafraniec, {\em Unbounded operators and
subnormality}, in preparation.
   \bibitem{sz1} F. H. Szafraniec, Kato-Protter type
inequalities, bounded vectors and the exponential
function, {\it Ann. Polon. Math.} {\bf 51} (1990),
303-312.
   \bibitem{vel-pan} T. Veluchamy, S. Panayappan,
Paranormal composition operators, {\em Indian J. Pure
Appl. Math.,} {\bf 24} (1993), 257-262.
   \bibitem{weid} J. Weidmann, {\it Linear operators in
Hilbert spaces}, Springer-Verlag, Berlin, Heidelberg,
New York, {\bf 1980}.
   \bibitem{wh} R. Whitley, Normal and quasinormal composition
operators, {\em Proc. Amer. Math. Soc.} {\bf 70}
(1978), 114-118.
   \end{thebibliography}
   
   \end{document}